\documentclass{mod}
\usepackage{url}
\usepackage{setspace}
\usepackage{scrextend}
\usepackage{cite}
\usepackage{multirow}
\usepackage{cancel}
\usepackage{amsmath}
\usepackage{amsthm}
\usepackage{amssymb}
\usepackage{yhmath}
\usepackage{mathtools}
\usepackage{mathrsfs}
\usepackage[all]{xy}
\usepackage[toc,page]{appendix}
\usepackage{titlesec}
\usepackage[nottoc]{tocbibind}
\usepackage[titles]{tocloft}
\usepackage{color}
\usepackage{tikz-cd}
\usepackage{titletoc}
\usepackage[dvipsnames]{xcolor}

\usepackage{indentfirst}

\usepackage{ stmaryrd }

\usepackage{subcaption}
\usepackage[shortlabels]{enumitem}

\newtheorem{thm}{Theorem}[section]

\newtheorem{lem}[thm]{Lemma}
\newtheorem{cor}[thm]{Corollary}
\newtheorem{prop}[thm]{Proposition}
\newtheorem{defn}[thm]{Definition}

\newtheorem{rem}[thm]{Remark}

\newcommand{\mathds}[1]{\text{\usefont{U}{dsrom}{m}{n}#1}}
\newcommand{\one}{\mathds{1}}
\newcommand{\m}{\mathfrak m}

\newcommand{\f}{\mathfrak f}

\newcommand{\C}{\mathbb C}

\newcommand{\id}{\mathrm{id}}

\newcommand{\Ev}{\mathrm{Ev}}

\newcommand{\ml}{\mathfrak l}

\newcommand{\q}{\mathfrak q}

\newcommand{\Map}{\mathrm{Map}}

\newcommand{\pt}{\mathrm{pt}}

\titleformat{\subsection}[runin]{
	\bfseries
	\itshape\normalsize}{(\thesubsection) }{0em}{}[\mbox{.} ]
\titlespacing{\section}
{0pt}
{2.25ex plus 1ex minus .2ex}
{10pt}

\titlecontents{section}
[0em]
{\scriptsize\vspace{1pt}}%
{\thecontentslabel.\enspace}
{}
{\titlerule*[0.5pc]{.}\contentspage}%

\allowdisplaybreaks

\title{\hspace{-1.2em}Open-Closed Hochschild Homology and the Relative Disk Mapping Space}

\author[Yi~Wang]{Yi~Wang}
\address{Yi Wang,
  Department of Mathematics, Purdue University, West Lafayette, Indiana 47907, United States}
  \email{wang6206@purdue.edu}

\author[Hang~Yuan]{Hang~Yuan}
\address{Hang Yuan, Beijing Institute of Mathematical Sciences and Applications, Beijing 101408, China}
\email{yuanhang@bimsa.cn}

\begin{abstract}
	{\sc Abstract.} 
	It is known that a model for the differential graded algebra (dga) of differential forms on the free loop space $LN$ of a simply connected smooth manifold $N$ is given by the Hochschild chain complex of the dga $\Omega(N)$ of differential forms on $N$, as shown by K.-T.~Chen via his theory of iterated integrals.
	We develop a relative version of Chen’s model.
	Given a smooth map $f\colon N\to M$ between smooth manifolds, we consider the ``relative disk mapping space'' consisting of pairs $(\Phi,\gamma)$ of maps $\Phi\colon \mathbb D\to M$ and $\gamma\colon S^1\to N$ such that $\Phi|_{\partial\mathbb D}=f\circ\gamma$.
	We construct iterated integral models for this mapping space through an open-closed homotopy algebra (OCHA) naturally associated to $f$ and the theory of open-closed Hochschild homology, which may be of independent interest.
	Our main theorem states that the resulting map is a quasi-isomorphism when $M$ is contractible or 2-connected with the rational homotopy type of an odd sphere, and $N$ is simply connected. This result generalizes Chen’s classical theorem for free loop spaces and, in the above special cases, extends the theorem of Getzler-Jones for double loop spaces.
\end{abstract}

\begin{document}
	\setlength{\parindent}{5.5mm}	\setlength{\parskip}{0em}

	\maketitle
	\setcounter{tocdepth}{1}
	\tableofcontents
	\section{Introduction}
	
	The free loop space $LN = \operatorname{Map}(S^1, N)$ of a smooth manifold $N$ is infinite dimensional, so the ordinary de~Rham theory cannot be applied to $LN$ directly. From a homotopy-theoretic viewpoint, $LN$ often has a highly intricate homotopy type, carrying far richer structures than $N$ itself; yet remarkably, its cohomology admits an algebraic model constructed entirely from the finite-dimensional manifold~$N$. The key bridge between these two worlds is Hochschild homology. This picture originates from K.\mbox{-}T.~Chen's theory of iterated integrals~\cite{chen1977iterated}: by integrating differential forms along paths/loops, Chen constructed a differential graded subalgebra of forms on the path/loop spaces, together with a natural cochain map from the bar construction on the algebra $\Omega(N)$ of differential forms.
	Specifically, on the algebraic side, iterated integrals assemble into the Hochschild chain complex (where $[1]$ denotes degree shifting down by 1)
	\[
	C(\Omega(N) )
	\;\cong\;
	\bigoplus_{k\ge 0}\Omega(N)\otimes \overline{\Omega(N)[1]}^{\otimes k}
	\]
	with the standard Hochschild differential $b$, and a natural chain map
	\begin{equation}
		\label{I_eq}
		I \colon \bigl(C(\Omega(N)),\,b\bigr)\longrightarrow \bigl(\Omega(LN),\,d\bigr)
	\end{equation}
	can be constructed via iterated integrals of differential forms on $N$.
	The idea is that one ``discretizes'' a loop $\gamma\colon S^1\to N$ by evaluating it at an ordered tuple of points on $S^1$, and the map $I$ sends a tensor $\omega_0\otimes\cdots\otimes\omega_k$ to the differential form on $LN$ obtained by integrating the pullbacks of the $\omega_i$.
	When $N$ is simply connected and of finite type, $I$ is indeed a quasi-isomorphism, so Hochschild homology of $\Omega(N)$ models the dga of differential forms on $LN$.
	Moreover, the circle action on $LN$ (rotating loops) corresponds to Connes’s $B$-operator on Hochschild chains, and cyclic homology of $\Omega(N)$ models the $S^1$-equivariant (co)homology of $LN$.

	In this article, we aim to explore a relative version of disk space / loop space: For a smooth map $f:N\to M$ between smooth manifolds, define the \textit{relative disk mapping space}
	\[
	\mathscr X:= \mathrm{Map}_f ( (\mathbb D, S^1), (M,N)) 
	\]
    which consists of pairs $(\Phi, \gamma)$ of smooth maps $\Phi: \mathbb D\to M$ and $\gamma:S^1\to N$ such that
	$
	\Phi(e^{2\pi i t}) = f(\gamma(t))
	$
	where we identify $[0,1]/\{0,1\}$ with $S^1$ via $t\mapsto e^{2\pi i t}$.

\smallskip 	
\smallskip 	
\smallskip 
\textbf{Question: }
	\textit{Is there a natural model for the dga of differential forms on $\mathscr X$?}
\smallskip 
\smallskip 	
\smallskip 	

	Several guiding ideas of the analysis of this question are as follows:
	
	\begin{itemize}
		\item[(i)] 
		One should attempt to build an algebraic model for the relative disk mapping space $\mathscr X$ from only $\Omega(M)$ and $\Omega(N)$, in direct analogy with the loop space case. Concretely, we seek canonical cochain maps
		\[
		\Omega(M)^{\otimes \ell}\otimes \Omega(N)^{\otimes k}\longrightarrow \Omega(\mathscr X)
		\]
		obtained by evaluating the loop $\gamma$ on an ordered ``discretizing'' tuple in $S^1$ as before and also evaluating the disk map $\Phi$ on an unordered sequence of points in $\mathbb D$.
		
		\item[(ii)] When $M$ is a point, $\mathscr X$ identifies with the free loop space $LN$.
		
		\item[(iii)] When $N$ is a point, $\mathscr X$ identifies with the based sphere space of $M$, or equivalently, the double based loop space of $M$.
	\end{itemize}
	
	For the starting point (i), the observation is that the smooth map $f:N \to M$ naturally gives a structure of open-closed homotopy algebra (OCHA) on the pair $(\Omega(M), \Omega(N))$. The notion of OCHA is introduced by Kajiura and Stasheff in \cite{kajiura2006homotopy} as a uniform generalization of both $L_\infty$ algebra and $A_\infty$ algebra.
	In particular, it is also a generalization of a dga.
	Then, we follow \cite{yuan2024open,yuan2025open} and develop in this paper an open-closed version of Hochschild chain complex on $\bigoplus_{\ell,k} \Omega(M)^{\otimes \ell} \otimes \Omega(N)^{\otimes k}$,
	which we denoted by $C(\Omega(M), \Omega(N))$. 
	Specifically, we can obtain the following result:
	
	\begin{thm}[Theorem \ref{J_cochain_map}]
		\label{main_thm}
		There is a natural differentiable space structure on $\mathscr X$ for which there is a canonical open-closed iterated integral cochain map
		\[
		J: C(\Omega(M); \Omega(N) ) \to \Omega(\mathscr X)
		\]
		induced by evaluation maps at interior points and boundary points. Moreover,:
		\begin{itemize}
			\itemsep 2pt
			\item When $M$ is a point, the map $J$ recovers Chen’s classical map for the free loop space of $N$.
			\item When $N$ is a point, the map $J$ is compatible, in the sense of Appendix \ref{s_compatible_GJ}, with Getzler-Jones' construction~\cite{getzler1994operads} for the double loop space $\Omega^2 M$.
		\end{itemize}
	\end{thm}

	In the setting of Theorem \ref{main_thm}, we expect $J$ to be a quasi-isomorphism under suitable hypotheses on the triple $(M,N,f)$. We have the following result.
	
	\begin{thm} [Theorem \ref{thm: J quasi isomorphism}]
		\label{main_thm_2}
		Assume $M$ is contractible or is 2-connected with the rational homotopy type of an odd-dimensional sphere, and $N$ is simply connected and of finite type. Then, the open-closed iterated integral map
		\[
		J: C(\Omega(M);\Omega(N)) \to \Omega(\mathscr X)
		\]
		in Theorem \ref{main_thm} is a quasi-isomorphism.
	\end{thm}

	\begin{rem}
The main results of this paper are that the map $J$ is a cochain map for every pair of smooth manifolds $M,N$ with a smooth map $f\colon N\to M$, and that it is a quasi-isomorphism in certain special cases.
Several directions for further investigation are outlined below:
		\begin{enumerate}
			\item \textbf{Additional structures.} Establish that $J$ preserves further algebraic structures. In particular, for the cyclic structure / $S^1$-action coming from the boundary loop, it should be straightforward to formulate and prove an analogue of \cite[Theorem 2.1]{getzler1991differential} in our context.
			
			\item \textbf{Quasi-isomorphism in broader settings.}  
            Extend the quasi-isomorphism result to more general pairs $(M,N)$.
            \begin{itemize}
                \item For general $2$-connected manifolds $M$, one would need a full generalization of the Getzler-Jones construction~\cite{getzler1994operads} to the relative setting. This appears challenging, since the cohomology ring of the relevant ``mixed configuration space'', consisting of points in $\mathbb D$ together with points on $S^1$ that are pairwise distinct, does not seem to admit a differential form model as simple as that of the usual configuration space of $\mathbb D$ (for the latter, see Appendix~\ref{section:Getzler Jones double loop space}).

                \item For the case where $N$ is not necessarily simply connected, the argument should extend along similar lines to those in~\cite[Chapter 2]{Wang2023thesis} and~\cite{RiveraWang2025PathSpaces}.
            \end{itemize}
			\item \textbf{Connections with factorization homology.}  
            Relate our construction to factorization homology for stratified spaces~\cite{ayala2017factorization}. This connection seems especially promising and may offer new insights to address the difficulties mentioned in (2).
			\item \textbf{Chain model with string topology operations and potential link to Lagrangian Floer theory.} The construction in this paper yields a Hochschild chain complex model for the cochains on the relative disk mapping space~$\mathscr X$. Alternatively, one may construct a Hochschild \emph{cochain} complex model for the \emph{chains} on~$\mathscr X$.
Combined with Irie’s de~Rham chain model for string topology~\cite{irie2020chain} and the first author’s cyclic extension~\cite{Wang2024}, this could lead to a de~Rham chain model of~$\mathscr X$ capable of encoding moduli spaces of holomorphic disks, thereby providing a possible bridge to Lagrangian Floer theory.
		\end{enumerate}
	\end{rem}

	This paper is organized as follows.
	
	In Section \ref{s_OCHA}, we review the notion of OCHA and develop the Hochschild chain complex and cyclic complex for a given unital OCHA.
	We also explore these complexes for the OCHA induced by a smooth map $f:N\to M$.
	In Section \ref{s_differentiable}, we review the notion of differentiable spaces and set up foundations for studying the iterated integral map for the relative mapping spaces.
	In Section \ref{s_smooth_fibration}, we present a construction of the de~Rham-Serre spectral sequence for smooth fibrations in the category of differentiable spaces.
	In Section \ref{s_relative_disk_space}, we specify the differentiable space structure on the relative disk mapping space $\mathscr X$, and apply the theories in Section \ref{s_differentiable} to construct the cochain map $J$ and prove Theorem \ref{main_thm}, with a careful computation of the signs.
	In Section \ref{section: quasi-isom results}, we apply results in Section \ref{s_smooth_fibration} and Appendix \ref{section:Getzler Jones double loop space} to prove Theorem \ref{main_thm_2}.
	In Appendix \ref{section:Getzler Jones double loop space}, we review Getzler-Jones' construction of an iterated integrals model for the double loop space and explain the compatibility with our construction.

\paragraph{Acknowledgement. }
The first author thanks Manuel Rivera for useful discussion and comments.
The second author thanks Jiahao Hu, Ryszard Nest, and Boris Tsygan for inspiring conversations, and Jim Stasheff for email correspondence and for his interest in potential open-closed versions of cyclic and Hochschild homology.
Both authors would like to thank Ezra Getzler for helpful conversations.

	\section{Open-closed homotopy algebra}
	\label{s_OCHA}

	\subsection{Definition of OCHA}

	Let $\Bbbk$ be a fixed ground field.
	Let $A = \bigoplus_{d \in \mathbb{Z}} A^d$ be a $\mathbb{Z}$-graded vector space over $\Bbbk$. Let $d_A: A^\bullet \to A^{\bullet+1}$ be a \textit{differential}. The degree of $a \in A$ is denoted by $|a|$.
	For $\varphi: A^{\otimes k}\to A'$ is a multilinear map, we write $|\varphi|=p$ if $
	|\varphi(a_1,\dots, a_k)| =|\varphi| +|a_1|+\cdots +|a_k|$.
	Define $C^k(A,A')$ to be the space of multilinear maps $A^{\otimes k}\to A'$, and define $
	C^\bullet (A, A')=\prod_{k\geqslant 1} C^k(A,A')$.

	An \textbf{\textit{$A_\infty$ algebra}} is an element $\m=\{\m_k : k\geqslant 1\}$ in $C^\bullet (A, A)$ such that $|\m|=|\m_k|=1$, the first term $\m_{1}=d_A$ agrees with the differential, and
	$
	\sum  (-1)^\ast \ \m_{k_1} (x_1,\dots, \m_{k_2} (x_i,\dots, x_{i+k_2-1}), \dots, x_k ) =0
	$
	where $\ast = \sum_{j=1}^{i-1} |x_j|$.
	Here this may be not the most common sign convention, but one may realized this by taking the shifted degree; see \cite{yuan2024open,yuan2021family}.
	
	For two $A_\infty$ algebras $(A,\m)$ and $(A',\m')$, an \textit{$A_\infty$ homomorphism} from $(A,\m)$ to $(A',\m')$ is $\f=\{\f_k: k\geqslant 1\}$ in $C^\bullet (A, A')$ such that $|\f|=|\f_k|=0$ and 
	$ \sum 
	\m'_{r} \big(
	\f_{j_1-j_0} (x_1,\dots) , \dots , \f_{j_\ell-j_{\ell-1}} (\dots, x_k) \big)
	=
	\sum (-1)^\ast \ \f_{\lambda+\mu+1} ( x_1,\dots, x_{\lambda}, \m_{\nu} (x_{\lambda+1},\dots, x_{\lambda+\nu}),\dots, x_k )
	$
	where $\ast=\sum_{j=1}^\lambda |x_j|$.
	
	Define $[k] := \{1, \dots, k\}$. The symmetric group of permutations on $[k]$ is denoted by $S_k$. Write
	\[
	I_1 \sqcup \cdots \sqcup I_r = [k]
	\]
	for a partition of $[k]$ into \textit{ordered} subsets $I_j = \{i_1 < i_2 < \cdots\}$ for $1 \leq j \leq r$.
	We also introduce the following ``dotted'' partition
	\begin{equation}
		\label{dot_sqcup_notation_eq}
		I_1 \dot\sqcup \cdots \dot\sqcup I_r = [k]
	\end{equation}
	to indicate a partition $I_1 \sqcup \cdots \sqcup I_r = [k]$ such that the condition that all elements in $I_i$ are smaller than those in $I_j$ whenever $i < j$.
	Besides, we introduce
	\begin{equation}
		\label{x_I_eq}
		x_I= x_{i_1}\otimes \cdots \otimes  x_{i_n}
	\end{equation}
	and write $|x_|=|x_{i_1}|+\cdots+|x_{i_n}|$
	for an ordered subset $I=\{i_1< \cdots <i_n\}$.
	Then, we may concisely write the $A_\infty$ associativity relation as $\sum (-1)^{|x_I|} \m (x_I , \m(x_J), x_K)=0$
	where the summation runs through $I\dot\sqcup J\dot\sqcup K=[k]$.

	A multilinear map $f: Z^{\otimes \ell}\to Z'$ of graded vector spaces $Z,Z'$ is called \textit{graded symmetric} if 
	$
	f(y_{\sigma(1)},\dots, y_{\sigma(\ell)})=(-1)^{\epsilon(\sigma)} f(y_1,\dots,y_\ell)
	$
	for $y_i\in Z$ and all permutations $\sigma \in S_\ell$.
	Here ${\epsilon(\sigma)} = \epsilon(\sigma; y_1,\dots, y_\ell) = \sum_{i<j: \sigma(i)>\sigma(j)} |y_i| \cdot |y_j|$.
	Let $I_{\ell}$ be the subspace of $Z^{\otimes \ell}$ generated by 
	$
	y_1\otimes \cdots \otimes y_\ell - (-1)^{\epsilon(\sigma)} y_{\sigma(1)}\otimes \cdots \otimes y_{\sigma(\ell)}
	$
	for $y_i\in Z$, $\sigma \in S_\ell$.
	
	Set 
	\[
	Z^{\wedge \ell} := Z^{\otimes \ell} / I_{\ell}.
	\]
	Specifically, $Z^{\wedge \ell}$ can be characterized by the following universal property: there is a canonical graded symmetric multilinear map $\varphi: Z^{\times \ell} \to Z^{\wedge \ell}$ such that for every graded symmetric multilinear map $f: Z^{\times \ell} \to E$, there is a unique linear map $f_{\wedge}: Z^{\wedge\ell}\to E$ with $f(y_1,\dots, y_\ell)=(-1)^{\epsilon(\sigma)} f_{\wedge}(\varphi(y_1,\dots, y_\ell))$.
	We write $y_1\wedge\cdots \wedge y_\ell$ for the image $\varphi(y_1,\dots, y_\ell)$. Then, we have
	$		y_1\wedge \cdots \wedge y_\ell = (-1)^{\epsilon(\sigma)}y_{\sigma(1)}\wedge \cdots \wedge y_{\sigma(\ell)}
	$.
	Slightly abusing the notation, we write 
	\begin{equation}
		\label{y_J_eq}
		y_J=y_{j_1}\wedge \cdots \wedge y_{j_n}
	\end{equation}
	for an ordered subset $J=\{j_1<\cdots <j_n\}$.
	
	An $(\ell_1,\ell_2)$-\textit{unshuffle} of $y_1,\dots, y_\ell$ with $\ell=\ell_1+\ell_2$ is a permutation $\sigma$ such that for $1\leqslant i<j \leqslant \ell_1$, we have $\sigma(i)<\sigma(j)$ and similar for $\ell_1+1\leqslant i<j\leqslant \ell$.
	This is equivalent to a partition of $[\ell]=\{1,\dots,\ell\}$ into two ordered subsets $J_1\sqcup J_2=[\ell]$ with $|J_1|=\ell_1$ and $|J_2|=\ell_2$. The sign $\epsilon = \epsilon(\sigma)$ of this unshuffle can be described as 
	\begin{equation}
		\label{koszul_eq}
		y_{[\ell]}=(-1)^{\epsilon} \ y_{J_1}\wedge y_{J_2}.
	\end{equation}

	Let $(Z,d_Z)$ and $(Z', d_{Z'})$ be differential graded vector spaces.
	Define $\tilde C^\ell (Z, Z')=\mathrm{Hom} (Z^{\wedge \ell} , Z')$ to be the space of graded symmetric multilinear maps $Z^{\wedge \ell}\to Z'$.
	Define 
	$
	\tilde C^\bullet (Z, Z') =\prod_{\ell\geqslant 1} \tilde C^\ell (Z, Z')$.
	
	An \textbf{\textit{$L_\infty$ algebra}} is an element $\ml=\{\ml_\ell: \ell\geqslant 1\}$ in $\tilde C^\bullet(Z, Z)$ such that $|\ml|=|\ml_{\ell}|=1$, $\ml_{1}=d_Z$, and 
	\[
	\sum_{J_1\sqcup J_2=[\ell]}  (-1)^\epsilon \ \ml_{|J_2|+1} ( \ml_{|J_1|} (y_{J_1}) \wedge y_{J_2} ) =0,
	\]
	with $\epsilon$ given by $y_{[\ell]}=(-1)^{\epsilon} y_{J_1}\wedge y_{J_2}$.

	Fix differential graded vector spaces $(A, d_A)$ and $(Z, d_Z)$. Define $C^{\ell, k} (Z ; A, A)$ to be the space of maps $\varphi: Z^{\wedge \ell}\otimes A^{\otimes k} \to A$.
	Define 
	\begin{equation}
		\label{hochschild_cochain_open_closed_eq}
		C^{\bullet,\bullet}(Z; A, A)= \prod_{\substack{\ell,k \geqslant 0 \\ (\ell,k)\neq (0,0)}} 
		C^{\ell, k} (Z; A, A).
	\end{equation}
	
	An element $\ml=\{\ml_\ell : \ell \geqslant 1\}$ in $\tilde C^\bullet (Z,Z)$ gives rise to a ``closed string action'' map
	\begin{equation}
		\label{widehat_ml_action_eq}
		\widehat{\ml}: C^{\bullet,\bullet}(Z; A, A) \to C^{\bullet,\bullet}(Z; A, A)
	\end{equation}
	defined as follows: for $D\in C^{\ell,k}(Z;A,A)$, $y_{[\ell]}=y_1\wedge \cdots \wedge y_\ell\in Z^{\wedge \ell}$, and $x_{[k]}=x_1\otimes \cdots \otimes x_k \in A^{\otimes k}$, we have
	\[
	\widehat{\ml}(D)_{\ell,k} (y_{[\ell]}; x_{[k]}) = \sum (-1)^\epsilon \  D_{|J_2|+1,k} (\ml_{|J_1|} (y_{J_1})\wedge y_{J_2}; x_{[k]}),
	\]
	where the summation runs through the partitions $[\ell]=J_1\sqcup J_2$ and the sign $\epsilon$ is given by $y_{[\ell]}=(-1)^\epsilon y_{J_1}\wedge y_{J_2}$.
	Whenever $\ml$ is an $L_\infty$ algebra, $\widehat{\ml} \ \ \widehat{\ml}=0$.
	
	Define the open-closed type {\textit{Gerstenhaber product}} $D\{E\}$ of $D\in C^{\bullet,\bullet} (Z; A, A)$ and $E\in C^{\bullet,\bullet} (Z; A, A)$ as an element in $C^{\bullet,\bullet} (Z; A, A)$ whose components are described as follows: 
	\begin{equation}
		\label{Gerstenhaber_prod_eq}
		(D\{E\})_{\ell,k} (y_{[\ell]}; x_{[k]}) =
		\sum
		(-1)^\ast  \ D_{|L_1|, |K_1|+|K_3|+1} (y_{L_1} ; x_{K_1} , E_{|L_2|, |K_2|} (y_{L_2} ; x_{K_2}), x_{K_3} ).
	\end{equation}
	Here the summation is
	$[\ell]=L_1\sqcup L_2$, $[k]=K_1\dot\sqcup K_2\dot\sqcup K_3$; the sign is 
	$\ast =|x_{K_1}|\cdot |y_{L_2}|+ \big(|y_{L_1}|+|x_{K_1}|\big) |D|
	+\epsilon$ with $\epsilon$ defined by $y_{L_1}\wedge y_{L_2}= (-1)^{\epsilon} y_{[\ell]}$.

	An \textbf{\textit{OCHA} (\textit{open-closed homotopy algebra})} is defined as a tuple $(Z, A, \ml, \q)$ that consists of an $L_\infty$ algebra $(Z,\ml)$ with $\ml \in \tilde C^\bullet(Z, Z)$
	and an element $\q=\{\q_{\ell,k}: \ell,k\geqslant 0, (\ell,k)\neq (0,0) \}$ in $C^{\bullet, \bullet}(Z; A, A)$ such that $|\q|=|\q_{\ell,k}|=1$, $\q_{0,1}=d_A$, and
	\[
	\q\{\q\} = \widehat{\ml} (\q).
	\]
	See also \cite[Definition 4.2]{yuan2024open}.

	\subsection{Open-closed Hochschild and cyclic homology}
	\label{s_OCHA_Hochschild}

	An OCHA $(Z, A, \q, \ml)$ is called \emph{c-unital} if there is an element $\one_A\in A$, called a \textit{c-unit}, such that $|\one_A|=-1$ and
	\begin{itemize}
		\itemsep 0pt
		\item[-] $\q_{0,2}(\one_A,x)=(-1)^{|x|-1}\q_{0,2}(x,\one_A)=x$;
		\item[-] $\q_{\ell,k}(\dots; \dots, \one_A, \dots)=0$ for all $(\ell,k)\neq (0,2)$. 
	\end{itemize}
	
	A c-unital OCHA $(Z, A, \q, \ml, \one_A)$ is called \emph{oc-unital} if there is an element $\one_Z\in Z$, called a \textit{c-unit}, such that $|\one_Z|=-2$ and
	\begin{itemize}
		\itemsep 0pt
		\item[-] $\q_{1,0}(\one_Z) = \one_A$;
		\item[-] $\q_{\ell,k}(\dots,\one_Z,\dots;\dots)=0$ for all $(\ell,k)\neq (1,0)$;
		\item[-] $\ml_\ell(\dots,\one_Z,\dots)=0$ for all $\ell\geq1$. 
	\end{itemize}
	
	Given an oc-unital OCHA $(Z;A,\ml,\q,\one_Z,\one_A)$, we set
	\[
	\overline Z := Z/(\Bbbk \one_Z), \quad \overline A := A/(\Bbbk \one_A),
	\]
	the quotients by the subspaces spanned by the units.
	For $\ell,k\geqslant 0$, we define
	\[
	C_{\ell,k} (Z; A) = \left(\overline Z^{\wedge \ell}\otimes A \otimes  \overline A^{\otimes k}\right)[-1],
	\]
	where $[-1]$ denotes degree shifting up by 1. Namely, the degree of an element 
	\[
	z_{[\ell]}\otimes a_{[0,k]}=z_1\wedge \cdots\wedge z_\ell \otimes a_0\otimes a_1\otimes \cdots \otimes a_k
	\]
	in $C_{\ell,k} (Z; A)$ is assigned to be
	\[
	|z_{[\ell]}|+|a_{[0,k]}|+1=\sum_{i=1}^\ell |z_i| +\sum_{i=0}^k |a_i|+1.
	\]
	Then, we define 
	$
	C_p(Z;A)
	$ to be the subspace of $\bigoplus_{\ell,k=0}^\infty C_{\ell,k}(Z; A)$ consisting of elements of degree $p$.
	Then, we set
	\[
	C(Z;A)=C_\bullet(Z;A) := \bigoplus_{p=-\infty}^{+\infty} C_p(Z;A)  \ \equiv \bigoplus_{\ell,k=0}^\infty C_{\ell,k}(Z; A).
	\]
	Given $z_1,\dots, z_\ell\in Z$, $K\subseteq [\ell]$ and a partition $L_1\sqcup \cdots L_N =K$, let $\epsilon (L_1,\dots, L_N; K)$ be a number in $\mathbb Z/2\mathbb Z$ such that
	\[
	(-1)^{\epsilon (L_1,\dots, L_N; K)} \ z_K = z_{L_1}\wedge \cdots \wedge z_{L_N}.
	\]

	For $i=0,1,2$, define degree 1 linear maps
	\[
	b_i: Z^{\wedge\ell}\otimes A\otimes A^{\otimes k} \to \bigoplus_{\ell'\leq \ell,\;k'\leq k+1} Z^{\wedge\ell'}\otimes A\otimes A^{\otimes k'}
	\]
	by
	\begin{equation*}
		\begin{aligned}
			b_0 (z_1\wedge \cdots \wedge z_\ell \otimes a_0 \cdots a_k) &=
			\sum_{J_0\sqcup J_1=[\ell]}
			(-1)^{\epsilon(J_0,J_1; [\ell])+1} \ \ml(z_{J_0})\wedge z_{J_1}\otimes a_0 \cdots a_k \\
			b_1 (z_1\wedge \cdots \wedge z_\ell \otimes a_0 \cdots a_k) &=
			\sum_{J_0\sqcup J_1=[\ell]  } \sum_{\lambda,\mu\geqslant 0}
			(-1)^{\ast_1} \ z_{J_0} \otimes a_0  \otimes \cdots  \otimes \q(z_{J_1} ; a_{\lambda+1},\dots, a_{\lambda+\mu}) \otimes  \cdots \otimes a_k) \\
			b_2 (z_1\wedge \cdots \wedge z_\ell \otimes a_0 \cdots a_k) &=
			\sum_{J_0\sqcup J_1=[\ell]}
			\sum_{0\leqslant i \leqslant  j \leqslant k }
			(-1)^{\ast_2} \ z_{J_0} \otimes \q( z_{J_1}; a_{j+1} ,\dots, a_k, a_0 ,a_1,\dots, a_{i} ) \otimes a_{i+1} \otimes \cdots \otimes a_{j},
		\end{aligned}
	\end{equation*}
	where the signs are
	\begin{align*}
		\ast_1
		&=\epsilon(J_0,J_1; [\ell])+|z_{J_0}|+|a_{[0,\lambda]}|+|z_{J_1}||a_{[0,\lambda]}|, \\
		\ast_2
		&=\epsilon(J_0,J_1; [\ell])+ |z_{J_0}|+ |a_{[j+1,k]}||a_{[0,j]}|.
	\end{align*}
	Clearly $b_0,b_1,b_2$ preserve the subspace spanned by elements $z_1\wedge\dots\wedge z_\ell\otimes a_0\cdots a_k$ where some $z_i=\one_Z$ ($i\geq 0$) or some $a_j=\one_A$ ($j\geq 1$), and hence pass to the quotient and induce linear maps
	\begin{equation}
		\label{b_i_defn_eq}
		b_i: C_\bullet (Z; A) \to C_{\bullet +1}(Z; A), \qquad i=0,1,2.   
	\end{equation}

	We also define an open-closed version of Connes's boundary map
	\[
	B : C_\bullet(Z; A) \to C_{\bullet-1}(Z; A)
	\]
	by 
	\[
	B(z_{[\ell]}\otimes a_{[0,k]}) = \sum_{i=0}^k (-1)^{|a_{[i+1,k]}||a_{[0,i]}|-|z_{[\ell]}|} 
	\ z_{[\ell]} \otimes \one_A \otimes a_{i+1} \otimes \cdots\otimes a_k \otimes a_0\otimes \cdots\otimes a_{i} .
	\]
	Since $|\one_A|=-1$, we see that $B$ is a map of degree $-1$.
	
	\begin{prop}
		\label{b_B_differential_prop}
		The maps
		\[
		b=b_0+b_1+b_2 : C_{\bullet} (Z; A) \to C_{\bullet +1}(Z; A)
		\quad \text{and} \quad
		B: C_\bullet (Z; A) \to C_{\bullet-1} (Z; A)
		\]
		satisfy that $b^2=B^2=bB+Bb=0$.
	\end{prop}
	
	\begin{proof}
		We can check that 
		\[
		b_1^2 = \sum (-1)^{\epsilon_{1}} \ z_{J_0} \otimes a_0 \otimes  \cdots \q\{\q\}(z_{J_1}; a_{\lambda+1},\dots, a_{\lambda+\mu}) \cdots \otimes a_k 
		\]
		where $\epsilon_1= \epsilon_0 + |z_{J_1}||a_{[0,\lambda]}|$ with $z_{[\ell]}=(-1)^{\epsilon_0}z_{J_0}\wedge z_{J_1}$.
		Observing the patterns
		$
		b_2^2 = \sum \pm z_{J_0}\otimes  \q (z_{J_1}; \dots \q( z_{J_2}; \dots, a_0,\dots) \dots) \dots 
		$ and $
		b_1b_2+b_2b_1= \sum \pm z_{J_0} \otimes \q(z_{J_1}; \dots , a_0,\dots, \q(z_{J_2}; \dots) \dots) \dots   \pm z_{J_0} \otimes \q(z_{J_1}; \dots, \q(z_{J_2}; \dots) , \dots , a_0, \dots) \dots 
		$,
		one can verify that
		\[
		b_2^2+b_1b_2+b_2b_1= \sum (-1)^{\epsilon_2} \ z_{J_0}\otimes \q\{\q\} (z_{J_1} ; a_{j+1} \dots a_0 \dots ) \dots a_{j}
		\]
		where $\epsilon_2=\epsilon_0+|a_{[j+1,k]}||a_{[0,j]}|$ with $\epsilon_0$ defined in the same way as above.
		Moreover, one can compute
		\begin{align*}
			b_1b_0+b_0b_1&= - \sum (-1)^{\epsilon_3} z_{J_0}\otimes a_0\otimes \cdots \widehat{\ml}(\q)(z_{J_1}; a_{\lambda+1}\dots a_{\lambda+\mu})\cdots \otimes a_k
			\\
			b_2b_0+b_0b_2&= - \sum (-1)^{\epsilon_4} z_{J_0} \otimes \widehat{\ml}(\q) (z_{J_1}; a_{j+1} \dots a_0 \dots ) \dots a_{j}
		\end{align*}
		where
		$\epsilon_3=\epsilon_0+|z_{J_1}||a_{[0,\lambda]}|=\epsilon_1$ and
		$\epsilon_4=\epsilon_0+|a_{[j+1,k]}||a_{[0,j]}|=\epsilon_2$.
		By the OCHA condition $\q\{\q\}=\widehat{\ml}(\q)$, we see that $b_1^2+b_1b_0+b_0b_1=0$ and $b_2^2+b_1b_2+b_2b_1+b_2b_0+b_0b_2=0$; hence, $b^2=(b_2+b_1+b_0)^2=0$.
		Similar to the $A_\infty$ case, one can verify that $(b_1+b_2)B+B(b_1+b_2)=0$ and $B^2=0$.
		It remains to check $b_0^2=0$ and $Bb_0+b_0B=0$ which are straightforward. 
	\end{proof}

	\begin{defn}
		We call $C_\bullet(Z;A)$ the \textit{open-closed Hochschild chain complex} of the open-closed homotopy algebra $(Z;A,\q,\ml)$. 
		We call $b$ the \textit{open-closed Hochschild chain differential}.
		The cohomology of $(C_\bullet(Z;A), b)$ is called its \textit{open-closed Hochschild homology}, denoted by $HH_\bullet (Z;A)$.
		If the context is clear, we often omit saying ``open-closed'' in these terminologies.
	\end{defn}	
	
	\subsection{Example from a smooth map between manifolds}
	\label{s_manifold_ocha}
	
	We study a special case of oc-unital OCHA in \cite[Example 4.5]{yuan2024open}.
	Let $f:N\to M$ be a smooth map between manifolds. We set $Z=\Omega(M)[2]$ and $A=\Omega(N)[1]$ to be the space of smooth differential forms on $M$ and $N$, with degree shifted down by 2 and 1, respectively.
	Namely, the \textit{shifted degrees} are given by:
	\begin{equation}
		\label{shifted_degree_eq}
		|\omega|=\deg(\omega)-2,\quad \omega\in\Omega(M)[2];\qquad |\eta|=\deg(\eta)-1, \quad \eta\in\Omega(N)[1],
	\end{equation}
	where $\deg(\omega)$, $\deg(\eta)$ stand for the internal degrees of differential forms.
	Then, we put $\ml_1=d_M$ and $\q_{0,1}=d_N$ to be the de Rham differentials; put $\q_{0,2}$ to be the signed wedge product such that
	\begin{equation}
		\label{wedge_signed_eq}
		\q_{0,2}(\eta,\eta')=(-1)^{|\eta|-1} \eta\wedge \eta'.
	\end{equation}
	We also put $\q_{1,0}$ to be the pullback map $f^*: \Omega^*(M)[2]\to \Omega^*(N)[1]$; finally, we put all other $\ml_\ell$ and $\q_{\ell,k}$ to be zero. Remark that for the above shifted degrees, $|\ml|=|\ml_1|=|d_M|=1$ and $|\q|=|\q_{0,1}|=|\q_{1,0}|=|\q_{0,2}|=1$. The units $\one_{\Omega(M)[2]},\one_{\Omega(N)[1]}$ are the constant functions with value 1. 
	Hence, for this OCHA structure, we obtain the open-closed Hochschild complex:
	\[ C(\Omega(M),\Omega(N)) = C(\Omega(M)[2],\Omega(N)[1]).
	\]
	For ease of notation, the degree shifts (down by 2 and 1) are understood to be in effect, although we will not indicate them explicitly.
	
	Let us compute the aforementioned $b=b_0+b_1+b_2$ in (\ref{b_i_defn_eq}) in this special situation:
	We first have
	\begin{align*}
		b_0(\omega_{[\ell]}\otimes \eta_{[0,k]})  
		&=\sum_{i=1}^\ell (-1)^{|\omega_i||\omega_{[i-1]}|+1} \ d_M\omega_i\wedge \omega_{[\ell]\setminus\{i\}} \otimes \eta_{[0,k]} \\
		&= \sum_{i=1}^\ell  (-1)^{|\omega_{[i-1]}|+1} \  \omega_{[i-1]}\wedge d_M\omega_i \wedge \omega_{[i+1,\ell]} \otimes\eta_{[0,k]}  .
	\end{align*}
	To describe $b_1$ and $b_2$ more specifically, we introduce
	\begin{align*}
		b_{1,0}(\omega_{[\ell]}\otimes \eta_{[0,k]})   
		&=\sum_{r=1}^{\ell}\sum_{i=0}^{k}(-1)^{|\omega_{[\ell]\setminus\{r\}}|+|\eta_{[0,i]}|+|\omega_r|(|\omega_{[r+1,\ell]}|+|\eta_{[0,i]}|)}
		\omega_{[\ell]\setminus\{r\}}\otimes \eta_{[0,i]}\otimes f^*\omega_r\otimes\eta_{[i+1,k]} \\
		b_{1,1}(\omega_{[\ell]}\otimes \eta_{[0,k]})   
		&=
		\sum_{j=1}^k
		(-1)^{|\omega_{[\ell]}|+|\eta_{[0,j-1]}|}\omega_{[\ell]}\otimes \eta_{[0,j-1]}\otimes d_N\eta_j \otimes \eta_{[j+1,k]} \\
		b_{1,2}
		(\omega_{[\ell]}\otimes \eta_{[0,k]})   
		&=
		\sum_{j=1}^{k-1}(-1)^{|\omega_{[\ell]}|+|\eta_{[0,j]}|-1}\omega_{[\ell]}\otimes \eta_{[0,j-1]}\otimes (\eta_j\wedge \eta_{j+1})\otimes\eta_{[j+2,k]}
	\end{align*}
	and
	\begin{align*}
		b_{2,0}(\omega_{[\ell]}\otimes \eta_{[0,k]}) & = 0
		\\
		b_{2,1} (\omega_{[\ell]}\otimes \eta_{[0,k]}) &=
		(-1)^{|\omega_{[\ell]}|}\omega_{[\ell]}\otimes d_N\eta_0\otimes\eta_{[1,k]} \\
		b_{2,2} (\omega_{[\ell]}\otimes \eta_{[0,k]}) &= (-1)^{|\eta_k||\eta_{[0,k-1]}|+|\omega_{[\ell]}|+|\eta_{k}|-1}\omega_{[\ell]}\otimes (\eta_{k}\wedge \eta_{0})\otimes\eta_{[1,k-1]} + (-1)^{|\omega_{[\ell]}|+|\eta_0|-1}\omega_{[\ell]}\otimes (\eta_{0}\wedge \eta_{1})\otimes\eta_{[2,k]} \\
		&=: b_{2,2}' (\omega_{[\ell]}\otimes \eta_{[0,k]}) 
		+
		b_{2,2}''(\omega_{[\ell]}\otimes \eta_{[0,k]}),
	\end{align*}
	where we use both (\ref{wedge_signed_eq}) and (\ref{shifted_degree_eq}).
	In the end, one can directly verify that 
	\begin{equation}
		\label{b_i012_eq}
		\begin{aligned}
			b_1&=b_{1,0}+b_{1,1}+b_{1,2} \\
			b_2&=b_{2,0}+b_{2,1}+b_{2,2}  =  b_{2,0}+b_{2,1}+b_{2,2}'+b_{2,2}''.
		\end{aligned}
	\end{equation}
	
	\smallskip 
	\smallskip 
	Finally, we prove the following lemma which will be useful later.
	
	\begin{lem} \label{lemma: sub dga (M,N)}
		Assume $M$ is 2-connected and $N$ is 1-connected. There exist dg subalgebras
		\[
		\mathcal{A}(M) \subset \Omega(M) \quad \text{and} \quad \mathcal{A}(N) \subset \Omega(N)
		\]
		with all of the following properties: 
		
		\begin{enumerate}
			\itemsep 2pt
			\item $\mathcal{A}^0(M)=\mathbb R$, $\mathcal{A}^1(M)=0$, $\mathcal{A}^2(M)=0$, $\mathcal{A}^0(N)=\mathbb R$, $\mathcal{A}^1(N)=0$.         
			\item The inclusion maps $i_M:\mathcal{A}(M) \hookrightarrow \Omega(M)$ and $i_N:\mathcal{A}(N) \hookrightarrow \Omega(N)$ are quasi-isomorphisms.
			\item $f^*(\mathcal{A}(M))\subset \mathcal{A}(N)$, and hence $f^*\circ i_M = i_N \circ f^*|_{\mathcal{A}(M)}$.
		\end{enumerate}
		
		Consequently, the aforementioned OCHA structure on the pair $(\Omega(M),\Omega(N))$ naturally restricts to the pair $(\mathcal A(M), \mathcal A(N))$ such that the natural inclusion of open-closed Hochschild chain complexes is a quasi-isomorphism:
		\[
		i: C(\mathcal A(M);\mathcal A(N)) \xhookrightarrow{\simeq} C(\Omega(M);\Omega(N)).
		\]
	\end{lem}
	\begin{proof}
		Let $V$ be a linear complement of $d(\Omega^2(M))$ in $\Omega^3(M)$, and let $W$ be a linear complement of $d(\Omega^1(N))$ in $\Omega^2(N)$. Then the dg algebras $\mathcal{A}(M)$, $\mathcal{A}(N)$ given by
		\[
		\mathcal{A}^n(M) = \begin{cases}
			\mathbb R & n=0 \\
			0 & n=1,2 \\
			V &  n = 3 \\
			\Omega^n(M) & n>3
		\end{cases}, \qquad \mathcal{A}^n(N) = \begin{cases}
			\mathbb R & n=0 \\
			0 & n=1 \\
			W &  n = 2 \\
			\Omega^n(M) & n>2
		\end{cases}
		\]
		satisfy all the required properties.
		
		As for the second half, we first note that the inclusions $i_M$ and $i_N$ gives rise to natural maps
		\[
		i_{\ell,k} \ : \ \overline{\mathcal A(M)}^{\wedge \ell} \otimes \mathcal A(N) \otimes \overline{\mathcal A(N)}^{\otimes k} 
		\to
		\overline{\Omega(M)}^{\wedge \ell} \otimes \Omega(N) \otimes \overline{\Omega(N)}^{\otimes k}, \qquad \ell,k\geq0.
		\]
		By the property (3), these map intertwine the two Hochschild differentials and therefore induce a map $i$ between the two Hochschild complexes.
		
		We prove $i$ is a quasi-isomorphism by a standard spectral sequence argument. Define an increasing filtration $F^p$ ($p\geq0$) on $C:=C(\Omega(M);\Omega(N))$ by
		\[
		F^p C := \langle \omega_{[\ell]}\otimes \eta_{[0,k]} \mid 2\ell +k \le p \rangle.
		\] 
		For the Hochschild differential $b=\sum b_{i,j}$, one easily verifies that
		\begin{align*}
			(b_{1,0}+b_{1,2}+b_{2,2}) (  F^p) & \subseteq   F^{p-1} \\
			(b_0+b_{1,1}+b_{2,1})(F^p) & \subseteq F^p
		\end{align*}
		and in particular $b(F^p)\subset F^p$. 
		Since $F^{-1}C=0$, the filtration $F^p$ is bounded below. Since $C$ is a direct sum over $\ell,k\geq0$, we have $C=\bigcup_{p\geq0}F^pC$, i.e., $F^p$ is exhaustive. Then, by the classical convergence theorem \cite[Theorem 5.5.1]{weibel1994introduction}, the associated spectral sequence 
		then $\{E_r^{p,q}(\Omega), d_r^{p,q}\}$ converges to $H\big(C(\Omega(M);\Omega(N)), b\big)$.
		Note that only $b_0, b_{1,1}, b_{2,1}$ survive in the $E_0$ page, so the $E_1$ page is given by $C(H(M); H(N))$. 
		The filtration $F^p$ also restricts to $C(\mathcal{A}(M);\mathcal{A}(N))$ and induces a convergent spectral sequence with the same $E_1$ page by the property (2) . Then, by the standard comparison theorem \cite[Theorem 5.2.12]{weibel1994introduction}, $i$ is a quasi-isomorphism.
	\end{proof}

	\section{Differentiable spaces}
	\label{s_differentiable}
	
	\subsection{Preliminaries}
	Following K-T Chen in \cite{chen1977iterated}, we introduce the following framework. 
	
	\begin{defn}
		\label{differentiable_defn}
		A \textit{differentiable space} $\mathscr X$ is a set equipped with a family $\mathcal P$ of set maps, called \textit{plots}, satisfying
		\begin{enumerate}[(a)]
			\itemsep 2pt
			\item Every plot is a map of the type $\phi: U\to \mathscr X$, where $U$ is a convex set and $\dim U$ can be arbitrary.
			\item If $\phi: U\to \mathscr X$ is a plot and if $\theta:U' \to U$ is a smooth map, then $\phi\circ \theta$ is a plot.
			\item Each constant map from a convex set $U$ to $\mathscr X$ is a plot.
			\item Let $\phi:U\to M$ be a set map. If $\{U_i\}$ is an open covering of $U$ and if every restriction $\phi|_{U_i}$ is a plot, then $\phi$ itself is a plot.
		\end{enumerate}
	\end{defn}
	If the family $\mathcal P$ only satisfies (a)-(c) and does not satisfy the condition (d), then we say $\mathscr X$ is a \textit{predifferentiable space}.
	Every predifferentiable space $\mathscr X$ has an induced differentiable space structure whose plots are set maps $\phi:U\to \mathscr X$ such that there exists an open covering $\{U_i\}$ of $U$ with each $\phi|_{U_i}$ living in the family $\mathcal P$ for the predifferentiable space structure.

	\begin{defn}
		A \textit{$p$-form} $w$ on a differentiable space $\mathscr X$ is a rule to assign to every plot $\phi:U\to \mathscr X$ a $p$-form $w_\phi$ on $U$ satisfying: If $\theta:U'\to U$ is a smooth map, then
		\[
		w_{\phi\circ\theta} = \theta^* w_\phi.
		\]
		Moreover, let $\Omega^p(\mathscr X)$ be the set of $p$-forms on which we can obtain a graded commutative differential graded algebra structure through the formulas $(w_1+w_2)_\phi=w_{1\phi}+w_{2\phi}$, $(cw)_\phi=cw_\phi$, $(w_1\wedge w_2)_\phi=w_{1\phi}\wedge w_{2\phi}$, and $(dw)_\phi=dw_\phi$.
		Note that by \cite[p 835]{chen1977iterated}, one can define differential forms on predifferentiable spaces in the same way, and the de Rham complex of a predifferentiable space coincides with that of the induced differentiable space. 
	\end{defn}

	\begin{defn}
		Given two differentiable spaces $\mathscr X_1$ and $\mathscr X_2$, 
		the family of maps $\phi_1\times \phi_2:U_1\times U_2 \to \mathscr X_1\times \mathscr X_2$, where $\phi_1:U_1\to \mathscr X_1$ and $\phi_2:U_2\to\mathscr X_2$ are plots, first defines a predifferentiable space structure and then induces a differentiable space structure on $\mathscr X_1\times \mathscr X_2$. 
	\end{defn}

	
	
	Note that a convex set $U$ can be linearly embedded into $\mathbb R^n$ with $n=\dim U$. This offers $U$ with Euclidean coordinates $\xi=(\xi^1,\dots, \xi^n)$ such that every differential forms on $U$ can be expressed as
	\[
	\sum_{i_1<\cdots <i_p} f_{i_1\cdots i_p}(\xi)  \ d\xi^{i_1}\wedge \cdots \wedge d\xi^{i_p}
	\]
	where $f_{i_1\cdots i_p}$ are smooth functions on $U$.
	Note also that the product $V\times U$ of two convex sets $V$ and $U$ is also a convex set.

	\begin{defn}
		Given two differentiable space $\mathscr X_1$ and $\mathscr X_2$, a \textit{differentiable map} $F: \mathscr X_1\to\mathscr X_2$ is defined to be a set map $f$ such that for every plot $\phi:U\to \mathscr X_1$, the composition $F\circ \phi:U\to \mathscr X_2$ is a plot.
		Furthermore, such a differentiable map naturally induces \textit{pullback} homomorphisms 
		\[
		F^*:\Omega(\mathscr X_2)\to\Omega(\mathscr X_1).
		\]
	\end{defn}

	A basic example of a differentiable space is any smooth manifold $Q$ (possibly with boundary or corners), whose plots are smooth maps from convex sets to $Q$; see \cite[Example 1.2.1]{chen1977iterated}.
	In particular, given a differentiable space $\mathscr X$, the product $Q\times \mathscr X$ is a differentiable space.
	A smooth map between manifolds $f:Q_1\to Q_2$ is a differentiable map when manifolds are viewed as differentiable spaces.
	Moreover, given a general differentiable space $\mathscr X$, one can check $f\times \id_{\mathscr X}: Q_1 \times \mathscr X\to Q_2\times \mathscr X$ is a differentiable map.
	Abusing the notation, we will write the induced pullback as
	\begin{equation}
		\label{pullback_product_space_eq}
		f^*: \Omega(Q_2\times\mathscr X)\to \Omega(Q_1\times \mathscr X).
	\end{equation}

	\subsection{Form-valued functions}
	Assume that $Q$ is an \textit{oriented} smooth manifold of real dimension $\dim Q=\mu$.
	We introduce a generalization of form-valued functions on differentiable spaces in \cite[Definition 1.4.1]{chen1977iterated} as follows:
	
	\begin{defn}
		\label{valued_function_defn}
		A \textit{$\Omega^p(\mathscr X)$-valued function} on $Q$ is an element $u$ of $\Omega^p(Q\times \mathscr X)$ such that for every plot $\phi: U\to \mathscr X$ and local parameterization $\lambda:V\to Q$ from a convex set $V$ with $\dim V=\dim Q=\mu$, the $p$-form $u_{\lambda\times \phi}\in\Omega(V\times U)$ on the plot 
		\[
		\lambda\times \phi:V\times U\to Q\times \mathscr X
		\]
		is of the type
		\begin{equation}
			\label{valued_function_on_Q_eq}
			\sum_{i_1<\cdots<i_p} a_{i_1\cdots i_p} (s, \xi) \ d\xi^{i_1}\wedge \cdots \wedge d\xi^{i_p}
		\end{equation}
		where $s=(s^1,\dots, s^\mu)$ and $\xi=(\xi^1,\dots, \xi^n)$ are coordinates of $V$ and $U$ respectively and $a_{i_1\cdots i_p}(s,\xi)$ are smooth functions on $V\times U$.
	\end{defn}

	Since the dimension of $Q$ may be larger than 1, it is natural to introduce the following further generalization of Definiton \ref{valued_function_defn} and \cite[Definition 1.4.1]{chen1977iterated}:

	\begin{defn}\label{defn: form-valued functions}
		A \textit{$\Omega^p(\mathscr X)$-valued $q$-form} on $Q$ is an element $v$ of $\Omega^{p+q}(Q\times \mathscr X)$ such that for every plot $\phi:U\to\mathscr X$ and local parameterization $\lambda:V\to Q$, the $(p+q)$-form $v_{\lambda\times \phi}$ on the plot $\lambda\times \phi$ is of the form
		\begin{equation}
			\label{valued_form_on_Q_eq}
			\sum_{i_1<\cdots<i_p} \sum_{j_1<\cdots <j_q} a_{i_1\cdots i_p; j_1\cdots j_q} (s, \xi) \ ds^{j_1}\wedge \cdots\wedge ds^{j_q}\wedge d\xi^{i_1}\wedge \cdots \wedge d\xi^{i_p}
		\end{equation}
		where $s=(s^1,\dots, s^\mu)$ and $\xi=(\xi^1,\dots, \xi^n)$ are coordinates of $V$ and $U$ respectively and the coefficient functions $a_{i_1\cdots i_p; j_1\cdots j_q}(s,\xi)$ are smooth on $V\times U$.
		We say $v$ has \textit{compact support} if the coefficient functions have compact supports on $V\times U$ for the above class of plots.
	\end{defn}

	If $q=\mu-1$, we are interested in the concrete computation of 
	\begin{equation}
		\label{d_Q_X_v_eq}
		d_{Q\times \mathscr X} \ v.
	\end{equation}
	For a plot $\lambda\times \phi$ of the above type, the expression (\ref{valued_form_on_Q_eq}) of $v_{\lambda\times \phi}$ has a simpler form
	\[
	\sum_{k=1}^\mu  \sum_{i_1=1}^n  \cdots \sum_{i_p=1}^n \frac{a^k_{i_1\cdots i_p} (s,\xi)}{p!}  ds^1\wedge \cdots \wedge \widehat{ds^k}\wedge \cdots \wedge ds^\mu \wedge d\xi^{i_1}\wedge \cdots\wedge d\xi^{i_p}.
	\]
	Then, we obtain
	\begin{align*}
		& (d_{Q\times \mathscr X} \ v)_{\lambda\times\phi} = d(v_{\lambda\times \phi}) \\
		&= \sum \frac{1}{p!} \frac{\partial a_{i_1\cdots i_p}^k}{\partial s^k} (-1)^{k-1} ds^1\wedge \cdots \wedge ds^\mu \wedge d\xi^{i_1}\wedge \cdots\wedge d\xi^{i_p} \\
		&+
		\sum \frac{1}{p!} \frac{\partial a_{i_1\cdots i_p}^k}{\partial \xi^{i_0}} (-1)^{\mu-1} ds^1\wedge \cdots \wedge \widehat{ds^k}\wedge \cdots \wedge ds^\mu \wedge d\xi^{i_0} \wedge d\xi^{i_1}\wedge \cdots\wedge d\xi^{i_p}.
	\end{align*}

	If $q=\mu$, every $\Omega^p(\mathscr X)$-valued $\mu$-form $v$ on $Q$ gives rise to a natural element
	\begin{equation}
		\label{int_Q_v_eq}
		\int_Q v \in \Omega^p(\mathscr X)
	\end{equation}
	which is defined as follows:
	
	Observe first that
	given a smooth function $\chi$ on $Q$, the product $\chi \cdot u$ is also a $\Omega^p(\mathscr X)$-valued $\mu$-form on $Q$ for which the expression (\ref{valued_form_on_Q_eq}) will be changed to 
	\[
	\sum_{i_1<\cdots<i_p}  \tilde a_{i_1\cdots i_p} (s, \xi) \ ds^{1}\wedge \cdots\wedge ds^{\mu}\wedge d\xi^{i_1}\wedge \cdots \wedge d\xi^{i_p}
	\]
	with
	\[
	\tilde a_{i_1\cdots i_p} (s, \xi) = \chi (\lambda(s)) \ a_{i_1\cdots i_p} (s, \xi).
	\]
	This observation allows us to use a partition of unity $\{\chi_i\}$ subordinated to an open covering of $Q$. 
	Accordingly, by replacing $v$ with some $\chi_i\cdot v$, we may assume that $Q=\mathbb R^\mu$ and $v$ has compact supports.
	Let's first consider plots of the special form $\id\times \phi$ where $\id:\mathbb R^\mu\to\mathbb R^\mu$ is the identity.
	By (\ref{valued_form_on_Q_eq}), we define
	\[
	\left(\int_Q v
	\right)_{\id\times \phi}= \sum_{i_1<\cdots <i_p} 
	\sum_{j_1<\cdots <j_\mu}
	\left(
	\int_{\mathbb R^\mu}  a_{i_1\cdots i_p; j_1\cdots j_\mu}(s,\xi) \  ds^1\cdots ds^\mu
	\right) d\xi^{i_1}\wedge \cdots \wedge d\xi^{i_p},
	\]
	which is a differential $p$-form on $\mathbb R^\mu\times U$.
	Next, for a general plot $\lambda\times \phi:V\times U\to \mathbb R^\mu \times \mathscr X$, we define
	\[
	\left( \int_Q v \right)_{\lambda\times \phi}= (\lambda\times \id_U)^* \left(  \int_Q v \right)_{\id\times \phi}.
	\]

	\subsection{Integration along fibers}\label{section:integration along fibers}
	We first fix terminology related to smooth manifolds with corners.  
	Let $Q$ be a smooth manifold with corners of dimension $\mu$. For any point $q \in Q$, there exists a unique integer $n \in [0,\mu]$ such that there is a chart at $q$ of the form
	\[
	\varphi_\alpha: O_\alpha \xrightarrow{\cong} \mathbb{R}^{\mu-n} \times [0,\infty)^n,
	\]
	where $O_\alpha$ is an open neighborhood of $q$ in $Q$.

	Let $s_1, \dots, s_\mu$ denote the coordinates on $\mathbb{R}^\mu$.
	We define the boundary of $\mathbb{R}^{\mu-n} \times [0,\infty)^n$ by
	\begin{equation}
		\label{equation: boundary of charts corners}
		\partial(\mathbb{R}^{\mu-n} \times [0,\infty)^n) := \bigsqcup_{j=1}^n \mathbb{R}^{\mu-n} \times [0,\infty)^{j-1} \times \{0\} \times [0,\infty)^{n-j}.
	\end{equation}
	Accordingly, we set
	$\partial O_\alpha := \varphi_\alpha^{-1} \left( \partial\left( \mathbb{R}^{\mu-n} \times [0,\infty)^n \right) \right)$.
	Let $\{(O_\alpha, \varphi_\alpha)\}$ be an atlas of $Q$ consisting of such charts. Then the collection $\{ (\partial O_\alpha, \varphi_\alpha|_{\partial O_\alpha}) \}$ forms a compatible atlas for a smooth manifold with corners, which we call the \emph{boundary} of $Q$ and denote by $\partial Q$. This definition is independent of the choice of atlas and agrees with \cite[Definition~2.6]{joycecorner}.
	
	There is a natural immersion
	\[
	i_Q : \partial Q \to Q
	\]
	induced in local charts by the canonical map
	\[
	i_n = \bigsqcup_{j=1}^n i_{n,j} : \ \ \bigsqcup_{j=1}^n \mathbb{R}^{\mu-n} \times [0,\infty)^{j-1} \times \{0\} \times [0,\infty)^{n-j} \to \mathbb{R}^\mu,
	\]
	where each $i_{n,j}$ is the standard inclusion. Note that the preimage $i_n^{-1}(0)$ consists of $n$ distinct points, so $i_n$ fails to be injective whenever $n > 1$. Consequently, the global map $i_Q$ is not injective in general; it is injective if and only if $Q$ is a manifold with (ordinary) boundary.
	
	From now on, assume that $Q$ is oriented. Then its boundary $\partial Q$ inherits a canonical orientation, defined locally as follows.  
	On the local model $\mathbb{R}^{\mu - n} \times [0,\infty)^n \subset \mathbb{R}^\mu$, we use the standard orientation determined by the volume form $ds^1 \wedge \dots \wedge ds^\mu$. For each boundary stratum $\mathbb{R}^{\mu-n} \times [0,\infty)^{j-1} \times \{0\} \times [0,\infty)^{n-j}$ with $k:=\mu-n+j$, we define its orientation via contraction with the outward normal vector:
	\begin{equation}\label{equation: orientation on H_k}
		\left( -\frac{\partial}{\partial s^k} \right) \lrcorner\, \left( ds^1 \wedge \dots \wedge ds^\mu \right) = (-1)^{k} \, ds^1 \wedge \dots \wedge \widehat{ds^k} \wedge \dots \wedge ds^\mu.
	\end{equation}

	Now let $\mathscr{X}$ be a differentiable space. There is a straightforward higher-dimensional analogue of \cite[Equation~(1.4.1)]{chen1977iterated} in the following sense:
	every $p$-form $v$ on $Q \times \mathscr{X}$ can be uniquely written as
	\begin{equation}
		\label{decomposition_degrees_eq}
		v = v^{(0)} + v^{(1)} + \cdots + v^{(\mu-1)} + v^{(\mu)},
	\end{equation}
	where $v^{(j)}$ is a $\Omega^{p-j}(\mathscr{X})$-valued $j$-form on $Q$ for $0 \leqslant j \leqslant \mu$.

	In view of (\ref{int_Q_v_eq}) and (\ref{decomposition_degrees_eq}), we define an operator
	\begin{align}
		\label{int_Q_map_eq}
		\int_Q : \Omega^p(Q \times \mathscr{X}) &\to \Omega^{p - \mu}(\mathscr{X})\\\nonumber
		v &\mapsto \int_Q v := \int_Q v^{(\mu)},
	\end{align}
	which is called \emph{integration along fibers}.
	Since $\partial Q$ is an oriented smooth manifold with corners of dimension $\mu-1$, we also have
	
	\begin{align*}
		\int_{\partial Q} : \Omega^p(\partial Q\times \mathscr X)&\to \Omega^{p-\mu+1}(\mathscr X)\\\nonumber
		v &\mapsto \int_{\partial Q} v := \int_{\partial Q} v^{(\mu-1)}.
	\end{align*}
	As in (\ref{pullback_product_space_eq}), $i_Q:\partial Q\to Q$ induces a pullback
	\[
	i_Q^*: \Omega(Q\times \mathscr X)\to \Omega(\partial Q\times \mathscr X).
	\]

	\begin{thm}[Stokes' Formula for integration along fibers]
		\label{integration_along_fiber_thm}
		For any oriented smooth manifold with corners $Q$ of dimension $\mu$, and any differentiable space $\mathscr X$, we have
		\[
		(-1)^{\mu}\,d_{\mathscr X} \circ \int_Q = \int_Q \circ \ d_{Q\times \mathscr X} - \int_{\partial Q} \circ \ i_Q^*.
		\]
	\end{thm}
	
	\begin{proof}
		Fix an element $v$ in $\Omega^p(Q\times \mathscr X)$.
		Let's first assume that $v$ has compact support and $Q=\mathbb R^{\mu-n}\times [0,\infty)^n$ with coordinates $(s_1,\dots,s_\mu)$. Furthermore, let us first consider plots of the special type $\id_Q\times \phi: Q\times U\to Q\times \mathscr X$ where $\phi$ is a plot of $\mathscr X$.
		In view of Definition \ref{valued_form_on_Q_eq} and (\ref{decomposition_degrees_eq}), we may write
		\begin{align*}
			(v^{(\mu-1)})_{\id_Q\times \phi}&=\sum_{k=1}^\mu \sum_{i_1,\dots, i_p} \frac{a_{i_1\cdots i_p}^k}{p!} ds^1\wedge \cdots \wedge \widehat{ds^k}\wedge \cdots \wedge ds^\mu \wedge d\xi^{i_1}\wedge \cdots\wedge d\xi^{i_p} \\
			(v^{(\mu)})_{\id_Q\times \phi}&=
			\sum_{i_1,\dots, i_p} \frac{b_{i_1\cdots i_p}}{p!} ds^1\wedge \cdots \wedge ds^\mu \wedge d\xi^{i_1}\wedge \cdots\wedge d\xi^{i_p}
		\end{align*}
		for some coordinate functions $a$'s and $b$'s.
		In the following, we will extensively use the computations of (\ref{d_Q_X_v_eq}) and (\ref{int_Q_v_eq}).
		For the left-hand side of the formula, we have
		\begin{align*}
			\left( d_{\mathscr X}  \int_Q \ v
			\right)_{\id_Q\times \phi} 
			&= \left( d_{\mathscr X}  \int_Q \ v^{(\mu)}
			\right)_{\id_Q\times \phi}  \\
			&= \sum_{j,i_1,\dots, i_p} \frac{1}{p!} \left( \int_Q \frac{\partial b_{i_1\cdots i_p}(s,\xi)}{\partial \xi^j} \ ds^1\cdots ds^\mu \right) d\xi^j\wedge d\xi^{i_1}\wedge \cdots \wedge d\xi^{i_p}.
		\end{align*}
		For the first term on the right-hand side of the formula, we have
		\begin{align*}
			\left(
			\int_Q  \ d_{Q\times \mathscr X} \ v
			\right)_{\id_Q\times \phi}
			&= \left(
			\int_Q  \ d_{Q\times \mathscr X} \ v^{(\mu-1)}
			\right)_{\id_Q\times \phi}
			+ \left(
			\int_Q  \ d_{Q\times \mathscr X} \ v^{(\mu)}
			\right)_{\id_Q\times \phi} \\
			&=\sum_{i_1,\dots, i_p} \sum_k \frac{(-1)^{k-1}}{p!} \left(\int_Q \frac{\partial a^k_{i_1\cdots i_p}(s,\xi)}{\partial s^k} ds^1\cdots ds^\mu \right) d\xi^{i_1}\wedge \cdots \wedge d\xi^{i_p} \\
			&+
			\sum_{j,i_1,\dots,i_p} \frac{(-1)^\mu}{p!}
			\left( \int_Q \frac{\partial b_{i_1\cdots i_p}(s,\xi)}{\partial \xi^j} ds^1 \cdots ds^\mu 
			\right)
			d\xi^j\wedge d\xi^{i_1}\wedge \cdots \wedge d\xi^{i_p}.
		\end{align*}
		Thus, it remains to show
		\begin{equation}\label{equation: int partial Q}
			\sum_{i_1,\dots, i_p} \sum_k \frac{(-1)^{k-1}}{p!} \left(\int_Q \frac{\partial a^k_{i_1\cdots i_p}(s,\xi)}{\partial s^k} ds^1\cdots ds^\mu \right) d\xi^{i_1}\wedge \cdots \wedge d\xi^{i_p}=\left(\int_{\partial Q}  \ i_Q^* v \right)_{\id_Q\times \phi}.
		\end{equation}
		The left-hand side of \eqref{equation: int partial Q} is
		\begin{align*}
			&\sum_{i_1,\dots, i_p}\sum_{k>\mu-n} \frac{(-1)^{k-1}}{p!} \left( \int_{\mathbb R^{\mu-n}\times[0,\infty)^{n-1}}  \left(\int_0^\infty \frac{\partial a^k_{i_1\cdots i_p}(s,\xi)}{\partial s^k} ds^k \right) ds^1 \cdots ds^{\mu} \right) d\xi^{i_1}\wedge \cdots \wedge d\xi^{i_p} \\
			=\ &
			\sum_{i_1,\dots, i_p}\sum_{k>\mu-n} \frac{(-1)^{k}}{p!} \left(\int_{\mathbb R^{\mu-n}\times[0,\infty)^{n-1}} \bar{a}^k_{i_1\cdots i_p}(s,\xi) \ ds^1 \cdots\widehat{ds^k}\cdots  ds^{\mu}  \right) d\xi^{i_1}\wedge \cdots \wedge d\xi^{i_p},
		\end{align*}
		where 
		\[
		\bar{a}^k_{i_1\cdots i_p}(s,\xi):=a^k_{i_1\cdots i_p}(s_1,\dots,s_{k-1},0,s_{k+1},\dots,s_\mu,\xi).
		\]
		The right-hand side of \eqref{equation: int partial Q} is
		\begin{align*}
			& \left(\int_{\partial Q}  \ i_Q^* v^{(\mu-1)} \right)_{\id_Q\times \phi} 
			\\ =\ &
			\sum_{i_1,\dots, i_p}\sum_{j=1}^n \frac{1}{p!} \left(\int_{\mathbb{R}^{\mu-n} \times [0,\infty)^{j-1} \times \{0\} \times [0,\infty)^{n-j}} \bar{a}^k_{i_1\cdots i_p}(s,\xi) \ ds^1 \cdots\widehat{ds^{\mu-n+j}}\cdots  ds^{\mu}  \right) d\xi^{i_1}\wedge \cdots\wedge d\xi^{i_p}.
		\end{align*}
		Then \eqref{equation: int partial Q} holds true because of the orientation \eqref{equation: orientation on H_k}.
		For a general plot of the form $\lambda\times \phi$, we can verify the formula by applying the pullback $(\lambda\times \id_U)^*$.
		Finally, for a general $Q$, using a partition of unity yields the formula.
	\end{proof}

	\section{The de~Rham-Serre spectral sequence for smooth fibrations}
	\label{s_smooth_fibration}

	A \textit{smooth fibration} (in the category of differentiable spaces) is defined as a morphism $\pi:E\to B$ (smooth map) between differentiable spaces which has the \textit{smooth homotopy lifting property} (smooth HLP) with respect to all differentiable spaces: 
	For any differentiable space $Y$ and any smooth maps 
	\[H_0:Y\to E, \quad h:[0,1]\times Y\to B \quad \text{such that} \quad \pi\circ H_0=h(0,\cdot),\] 
	there exists a smooth map 
	\[H:[0,1]\times Y\to E \quad \text{such that} \quad H(0,
	\cdot)=H_0\ \text{ and }\ \pi\circ H = h.\]

	Let $\pi:E\to B$ be a smooth fibration between differentiable spaces.
	We further assume:
	
	\begin{itemize}
		\itemsep 2pt
		\item $E$ and $B$ carry topologies such that all plots are continuous.
		\item $B$ is (smoothly) path-connected.
		\item $B$ is locally contractible: every $b\in B$ has an open neighborhood $U_b$ which is (smoothly) contractible.
		\item $B$ is paracompact, Hausdorff, and admits a smooth partition of unity subordinate to any open cover. (This ensures that $\mathcal{C}^\infty_B$-modules are fine sheaves.)
		\item $\Omega^\bullet_B$ is a locally free (possibly infinite rank) $\mathcal{C}^\infty_B$-module.
	\end{itemize}
	Here $\mathcal{C}^\infty_B$ is the sheaf of smooth functions on $B$, and $\Omega^\bullet_B$ is the sheaf of differential forms on $B$.
	
	For $b\in B$, denote the fiber $E_b:=\pi^{-1}(b)$.

	\subsection{Local systems from smooth fibrations}
	We briefly review some standard facts. 
	
	For any points $b,b'\in B$, choose a smooth path $\gamma:[0,1]\to B$ from $b$
	to $b'$.  By the smooth HLP, $\gamma$ admits a smooth lift
	$\widetilde\gamma:[0,1]\times E_{b}\to E$ with
	$\pi\circ\widetilde\gamma(t,e)=\gamma(t)$.
	Then
	$\widetilde\gamma(1,\cdot):E_{b}\to E_{b'}$
	is a smooth homotopy equivalence of fibers, inducing an isomorphism on de~Rham cohomology groups:
	\[
	T_\gamma:= \widetilde\gamma(1,\cdot)^*: H^\bullet_{\mathrm{dR}}(E_{b'})\xrightarrow{\cong} H^\bullet_{\mathrm{dR}}(E_{b}).
	\]
	If $\gamma_0,\gamma_1$ are smoothly homotopic relative to endpoints, then
	$T_{\gamma_0}=T_{\gamma_1}$. For a smooth concatenation $\gamma_1 * \gamma_2$,
	$T_{\gamma_1 * \gamma_2}=T_{\gamma_1}\circ T_{\gamma_2}$. 
	
	Fix a basepoint $b_0\in B$ and denote $F=E_{b_0}$. Then
	$\gamma\mapsto T_\gamma$ for $\gamma$ based at $b$ induces the \emph{monodromy representation}
	\[
	\rho: \pi_1(B,b_0) \to \operatorname{GL}\!\big(H^\bullet_{\mathrm{dR}}(F)\big),
	\]
	and determines a \emph{local system of vector spaces} on $B$, denoted
	\[
	\mathcal H^q_{\mathrm{dR}}(F),
	\qquad\text{with stalk $\mathcal H^q_{\mathrm{dR}}(F)\big|_b=H^q_{\mathrm{dR}}(E_b)\cong H^q_{\mathrm{dR}}(F)$ at $b\in B$}.
	\]
	Since $B$ is locally contractible, $\mathcal{H}^q_{\mathrm{dR}}(F)$ is a locally constant sheaf on $B$.
	
	\subsection{The sheaf of relative differential forms}
	Denote by $\Omega^\bullet_E$ and $\Omega^\bullet_B$ the sheaf of differential forms on $E$ and $B$, and define the sheaf of ideals
	\[
	\mathcal{I} := \pi^*\Omega^1_B\wedge\Omega^{\bullet -1}_E \subset \Omega^\bullet_E.
	\]
	Then the quotient presheaf
	\[
	U \mapsto \Omega^\bullet_E(U)/\mathcal{I}(U),\quad U\subset E \text{ open}
	\]
	is a sheaf on $E$: the locality axiom is satisfied since belonging to $\mathcal{I}(U)$ is a local property (can be checked on stalks), and the gluing axiom is satisfied since $E$ admits smooth partitions of unity. Denote this sheaf by $\Omega^\bullet_{E/B}$ and call it the \emph{sheaf of relative differential forms on $E$}. The differential $d_E$ on $\Omega^\bullet_E$ induces a differential on $\Omega^\bullet_{E/B}$ which we denote by $d_{E/B}$.
	
	Consider the sheaf of complexes $\pi_*\Omega^\bullet_{E/B}$ on $B$ (the pushforward of $\Omega^\bullet_{E/B}$ via $\pi$), and its cohomology sheaf $\mathcal{H}^\bullet(\pi_*\Omega_{E/B})$ (after sheafification).
	
	\begin{prop}\label{prop: H E/B isom to C infty B H(F)}
		There is a natural isomorphism of sheaves 
		\[
		\mathcal{H}^\bullet(\pi_*\Omega_{E/B}) \cong \mathcal{C}^\infty_B\otimes\mathcal{H}^\bullet_{\mathrm{dR}}(F).
		\]
	\end{prop}
	
	The proof of Proposition \ref{prop: H E/B isom to C infty B H(F)} is based on several lemmas.
	
	For any open set $U\subset B$, we set $E_U:=\pi^{-1}(U)$ and denote 
	\[
	\Omega^\bullet(E_U/U) := \Gamma\big(U,\pi_*\Omega^\bullet_{E/B}\big),\qquad H^\bullet_{\mathrm{dR}}(E_U/U) := H^\bullet\big(\Omega(E_U/U),d_{E/B}\big).
	\]
	Then $U\mapsto H^\bullet_{\mathrm{dR}}(E_U/U)$ defines a presheaf on $B$.
	\begin{lem}\label{lem:H E_U/U section}
		For any open set $U\subset B$, there is a natural isomorphism
		\[\Gamma\big(U,\mathcal{H}^\bullet(\pi_*\Omega_{E/B})\big) = H^\bullet_{\mathrm{dR}}(E_U/U).
		\]
		Namely, the sheaf $\mathcal{H}^\bullet(\pi_*\Omega_{E/B})$ and the presheaf $U\mapsto H^\bullet_{\mathrm{dR}}(E_U/U)$ are the same.
	\end{lem}
	\begin{proof}
		By construction, for each $q\geq0$, $\pi_*\Omega^q_{E/B}$ is a $\mathcal{C}^\infty_B$-module, hence a fine sheaf. Since the functor of sections $\Gamma(U,-)$ is exact on complexes of fine sheaves, we have
		\[
		\Gamma\big(U,\mathcal{H}^q(\pi_*\Omega^\bullet_{E/B})\big) = H^q\big(\Gamma(U,\pi_*\Omega^\bullet_{E/B})\big).
		\]
		But the latter is $H^q_{\mathrm{dR}}(E_U/U)$ by definition.
	\end{proof}
	\begin{lem}\label{lem:compute relatibe de Rham cohom}
		For any open set $U\subset$ B and $b\in B$, consider the trivial fibration $\mathrm{pr}_U:U\times F\to U$. Then there is a natural isomorphism
		\[
		H^q_{\mathrm{dR}}(U\times F/U) \cong C^\infty\big(U,H^q_{\mathrm{dR}}(F)\big).
		\]
		
	\end{lem}
	
	\begin{proof}
		First, there is a canonical dg algebra isomorphism
		\begin{align*}
			\Omega^\bullet(U\times F/U) & \xrightarrow{\cong}  C^\infty\big(U,\Omega^\bullet(F)\big)\\
			[f\cdot \mathrm{pr}_F^*\eta] & \mapsto \big(u\mapsto f(u,\cdot)\eta\big),\qquad f\in C^\infty(U\times F),\ \eta\in\Omega^\bullet(F),
		\end{align*}
		where the differential on $C^\infty\big(U,\Omega^\bullet(F)\big)$ is induced by $d_F$ (and still denoted by $d_F$):
		\[
		(d_F\phi)(u)=d_F(\phi(u)),\quad \phi\in C^\infty\big(U,\Omega^\bullet(F)\big).
		\]
		It follows that
		\[
		H^q_{\mathrm{dR}}(U\times F/U)
		\cong 
		H^q\big(C^\infty\big(U,\Omega^\bullet(F))\big).
		\]
		
		Next, consider the sheaf of complexes $\mathcal{S}^\bullet=\mathcal{C}^\infty_U\otimes \underline{\Omega^\bullet(F)}$ on $U$, where $\underline{\Omega^\bullet(F)}$ denotes the constant sheaf of complexes  with stalk $\Omega^\bullet(F)$. Then $\mathcal{S}^\bullet$ is a fine sheaf of complexes. Since the functor of sections $\Gamma(U,-)$ is exact on fine sheaves of complexes, we obtain
		\[
		H^q\big(C^\infty\big(U,\Omega^\bullet(F))\big) = H^q(\Gamma(U,\mathcal{S}^\bullet) \cong \Gamma(U,\mathcal{H}^q(\mathcal{S}^\bullet)),
		\]
		where $\mathcal{H}^q(\mathcal{S}^\bullet)$ is the cohomology sheaf of $\mathcal{S}^\bullet$.
		
		Finally, since tensoring over $\mathbb R$ is exact, 
		$
		\mathcal H^q(\mathcal{S}^\bullet) \cong \mathcal{C}^\infty_U\otimes \underline{H^q_{\mathrm{dR}}(F)}
		$. Hence
		\[
		\Gamma(U,\mathcal{H}^q(\mathcal{S}^\bullet)) \cong \Gamma\big(\mathcal{C}^\infty_U\otimes \underline{H^q_{\mathrm{dR}}(F)}\big) = C^\infty\big(U,H^q_{\mathrm{dR}}(F)\big).
		\]
		This completes the proof.
	\end{proof}
	
	The next lemma says $\pi:E\to B$ is \emph{fiber homotopy trivial} over contractible open sets $U\subset B$.
	
	\begin{lem}\label{lem:fib-hom-triv}
		Let $U\subset B$ be a smoothly contractible open set. Fix $b\in U$ and set $F:=E_b$.
		Then the smooth fibration $\pi_U:E_U\to U$ is fiber homotopy equivalent (over $U$) to the trivial fibration $\mathrm{pr}_U:U\times F\to U$. Namely, there exist smooth maps
		\[
		\Phi:E_U\to U\times F, \qquad
		\Phi':U\times F\to E_U
		\]
		with 
		$
		\mathrm{pr}_U\circ\Phi=\pi_U$, $\pi_U\circ\Phi'=\mathrm{pr}_U$,
		and smooth homotopies
		\[
		H:[0,1]\times E_U\to E_U, \qquad H':[0,1]\times U\times F \to  U\times F
		\]
		with $\pi_U(H(t,e))=\pi_U(e)$, $\mathrm{pr}_U(H'(t,u,y))=u$ and 
		\[
		H(0,\cdot)=\mathrm{id}, \quad H(1,\cdot)=\Phi'\circ\Phi,\quad H'(0,\cdot)=\mathrm{id},\quad H'(1,\cdot)=\Phi\circ\Phi'.
		\]
	\end{lem}
	
	\begin{proof}
		Fix $0<\varepsilon\ll1$, and choose a smooth contraction $\bar c:[0,1]\times U\to U$ from $\mathrm{id}_U$ to the constant map $u\mapsto b$ which satisfies
		$\bar c(t,u)=u$ for $t\in[0,\varepsilon]$ and $\bar c(t,u)=b$ for $t\in[1-\varepsilon,1]$.
		Define
		\[
		c:[0,1]\times E_U\to U,\quad c(t,e):=\bar c(t,\pi_U(e)).
		\]
		By the smooth HLP, there exists a smooth lift $C:[0,1]\times E_U\to E_U$ with
		\[
		\pi_U\circ C=c,\quad C(t,e)=e\ (t\in[0,\varepsilon]),\quad C(t,e)=C(1,e)\ (t\in[1-\varepsilon,1]).
		\]
		Set
		\[
		\Phi(e):=\bigl(\pi_U(e),\,C(1,e)\bigr)\in U\times F.
		\]
		Similarly, let $i_b:F=E_b\hookrightarrow E_U$ be the inclusion and define
		\[
		c':[0,1]\times U\times F\to U,\quad c'(t,u,y):=\bar c(1-t,u).
		\]
		Again by smooth HLP, there is a smooth lift $C':[0,1]\times U\times F\to E_U$ with
		\[
		\pi_U\circ C'=c',\quad C'(t,u,y)=i_b(y)\ (t\in[0,\varepsilon]),\quad C'(t,u,y)=C'(1,u,y)\ (t\in[1-\varepsilon,1]).
		\]
		Define
		\[
		\Phi'(u,y):=C'(1,u,y)\in E_U.
		\]
		Then $\mathrm{pr}_U\circ\Phi=\pi_U$ and $\pi_U\circ\Phi'=\mathrm{pr}_U$ by construction.
		
		To produce $H$, first form the (already smooth) concatenation
		\[
		(C\#C'):[0,1]\times E_U\to E_U,\qquad
		(C\#C')(t,e)=
		\begin{cases}
			C(2t,e), & t\in[0,\frac12],\\[2pt]
			C'(2t-1,\Phi(e)), & t\in[\frac12,1],
		\end{cases}
		\]
		which is smooth at $t=\frac12$ because $C,C'$ are constant in $t$ near $0,1$.
		Now define a base homotopy
		\[
		g:[0,1]\times[0,1]\times E_U\to U,\qquad
		g(t,s,e):=
		\begin{cases}
			\bar c(2st,\pi_U(e)), & t\in[0,\frac12],\\[2pt]
			\bar c(2s-2st,\pi_U(e)), & t\in[\frac12,1].
		\end{cases}
		\]
		For $(t,s)\in\{0\}\times[0,1]\cup[0,1]\times\partial[0,1]$ one has 
		\[
		g(0,s,e)=g(t,0,e)=\pi_U(e),\qquad g(t,1,e)=\pi_U\big((C\#C')(t,e)\big).
		\]
		By the smooth HLP \emph{relative to the pair} $([0,1]\times U,\ \partial[0,1]\times U)$ (a standard consequence of HLP),
		there is a smooth lift
		\[
		G:[0,1]\times[0,1]\times E_U\to E_U,\qquad \pi_U\circ G=g,
		\]
		such that $G(0,s,e)=G(t,0,e)=e$ and $G(t,1,e)=(C\#C')(t,e)$.
		Set
		\[
		H(s,e):=G(1,s,e).
		\]
		Then $\pi_U\circ H=\pi_U$, $H(0,\cdot)=\mathrm{id}$ and $H(1,\cdot)=\Phi'\circ\Phi$.
		The construction of $H'$ is analogous, exchanging the roles of $C$ and $C'$.
	\end{proof}
	
	We call $\Phi: E_U \to U\times F$ as above a \emph{fiber homotopy trivialization} of $E$ over $U$.
	
	\begin{cor}\label{cor:relative-qi}
		Let $U,F,\Phi,\Phi',H,H'$ be as in Lemma \ref{lem:fib-hom-triv}. 
		Then $\Phi$ induces a cochain-homotopy equivalence
		\[
		\Omega^\bullet(E_V/V) \simeq \Omega^\bullet(V\times F/V),
		\]
		for any open set $V\subset U$.
	\end{cor}
	
	\begin{proof}
		Since $\Phi,\Phi',H,H'$ are all \emph{over $U$}, for any open subset $V\subset U$, 
		\[
		\Phi_V=\Phi|_{E_V}: E_V\to V\times F, \quad \Phi_V'=\Phi'|_{V
			\times F}: V\times F \to E_V
		\]
		provides a fiber homotopy trivialization of $E$ over $V$. Pullback by $\Phi_V$ and $\Phi_V'$ preserves the ideals generated by base forms:
		\[
		\Phi_V^*\big(\mathrm{pr}_V^*\Omega^1(V)\big)\subset \pi_V^*\Omega^1(V),
		\quad
		{\Phi_V'}^{*}\big(\pi_V^*\Omega^1(V)\big)\subset \mathrm{pr}_V^*\Omega^1(V),
		\]
		so $\Phi_V^*$ and ${\Phi_V'}^*$ descend to cochain maps
		\[
		\Phi_V^*:\Omega^\bullet(V\times F/V)\to \Omega^\bullet(E_V/V),
		\quad
		{\Phi_V'}^*:\Omega^\bullet(E_V/V)\to \Omega^\bullet(V\times F/V).
		\]
		Let $K_{H_V}$ (resp.\ $K_{H_V'}$) be the de~Rham homotopy operator associated with $H_V$ (resp.\ $H_V'$) on \emph{absolute} forms:
		\[
		K_{H_V}(\omega)=\int_0^1\iota_{\partial_t}H_V^*\omega\,dt\ \big(\omega\in\Omega(E_V)\big),
		\quad
		d_EK_{H_V}+K_{H_V}d_E=(\Phi_V'\circ\Phi_V)^*-\mathrm{id},
		\]
		and similarly for $H_V'$. Because $H_V$ and $H_V'$ are \emph{over $V$}, we have
		$\iota_{\partial_t}H_V^*(\pi_V^*\alpha)=0$ and $\iota_{\partial_t}{H_V'}^*(\mathrm{pr}_V^*\alpha)=0$ for all $\alpha\in\Omega^1(V)$, hence $K_{H_V}$ and $K_{H_V'}$ preserve the respective ideals generated by $\pi_V^*\Omega^1(B)$ and $\mathrm{pr}_V^*\Omega^1(B)$, and therefore descend to
		\[
		K_{H_V}:\Omega^\bullet(E_V/V)\to \Omega^{\bullet-1}(E_V/V),
		\quad
		K_{H_V'}:\Omega^\bullet(V\times F/V)\to \Omega^{\bullet-1}(V\times F/V),
		\]
		satisfying the homotopy identities on the relative complexes.
		Thus $\Phi_V^*,{\Phi_V'}^*$ are cochain homotopy inverse to each other on the relative complexes.
	\end{proof}

	\begin{proof}[Proof of Proposition \ref{prop: H E/B isom to C infty B H(F)}]
		Let $\{U_i\}_{i\in A}$ be a cover of $B$ consisting of contractible open sets. By Lemma \ref{lem:H E_U/U section}, Lemma \ref{lem:compute relatibe de Rham cohom} and Corollary \ref{cor:relative-qi}, the smooth HLP induces a collection of isomorphisms of sheaves
		\[
		\Psi_i: \mathcal{H}^q(\pi_*\Omega^\bullet_{E/B})\big|_{U_i} \xrightarrow{\cong} \mathcal{C}^\infty_{U_i}\otimes\mathcal{H}^q_{\mathrm{dR}}(F), \quad i\in A.
		\]
		These (local) isomorphisms $\Psi_i$ are compatible (again by HLP) and glue to the desired isomorphism of sheaves $\mathcal{H}^q(\pi_*\Omega^\bullet_{E/B}) \xrightarrow{\cong} \mathcal{C}^\infty_{B}\otimes\mathcal{H}^q_{\mathrm{dR}}(F)$.
	\end{proof}

	\subsection{Filtration on $\Omega(E)$ by the degree of base forms}
	
	Define a decreasing filtration of subcomplexes of $\Omega^\bullet(E)$ by
	\[
	\mathcal{F}^p\Omega^\bullet(E):=\sum_{n\ge p}\pi^*\Omega^n(B)\wedge\Omega^{\bullet-n}(E) = \pi^*\Omega^p(B)\wedge\Omega^{\bullet-p}(E), \quad p\geq 0.
	\]
	Then $\mathcal{F}^p$ is a (degreewise) bounded filtration on $\Omega(E)$, hence by the classical convegence theorem \cite[Theorem 5.5.1]{weibel1994introduction}, the associated spectral sequence $\{E_r^{p,q},d_r\}$ converges to $H^\bullet_{\mathrm{dR}}(E):=H^\bullet(\Omega(E))$.
	
	We now determine the first few pages of $\{E_r^{p,q},d_r\}$.
	
	\subsection*{The $E_0$-page}
	By definition,
	\begin{equation}
		\label{E_0_sm_fib_eq}
		\begin{aligned}
			E_0^{p,q} & = \mathcal{F}^p\Omega^{p+q}(E)/\mathcal{F}^{p+1}\Omega^{p+q}(E) \\
			& =\big(\pi^*\Omega^p(B)\wedge\Omega^q(E)\big) / \big(\pi^*\Omega^{p+1}(B)\wedge\Omega^q(E)\big) \\
			& \cong \pi^*\Omega^p(B) \wedge \Omega^q(E/B) \\
			& \cong \Omega^p(B)\otimes_{C^\infty(B)}\Omega^q(E/B) \\
			& = \Gamma\big(B,\Omega^p_B\otimes_{\mathcal{C}^\infty_B}\pi_*\Omega^q_{E/B}\big),
		\end{aligned}
	\end{equation}
	Under these canonical isomorphisms, the de Rham differential $d_E$ on $\Omega(E)$ induces the differential
	\[
	d_0:E_0^{p,q}\to E_0^{p,q+1}, \quad 
	d_0=1\otimes d_{E/B}.
	\]
	
	\subsection*{The $E_1$-page}
	By definition,
	\[
	E_1^{p,q} = H^q(E_0^{p,\bullet},d_0) = H^q\big(\Gamma(B,\Omega^p_B\otimes_{\mathcal{C}^\infty_B}\pi_*\Omega^\bullet_{E/B})\big)
	\]
	Since $\Omega_B^p$ is a locally free $\mathcal{C}^\infty_B$-module sheaf and each
	$\pi_*\Omega^r_{E/B}$ is a fine $\mathcal{C}^\infty_B$-module, the functor
	$\Gamma(B,-)$ is exact on all the terms. Hence
	\[
	E_1^{p,q}
	= \Gamma\big(B,\Omega_B^p\otimes_{\mathcal{C}^\infty_B} \mathcal H^q(\pi_*\Omega^\bullet_{E/B})\big).
	\]
	Then by Proposition \ref{prop: H E/B isom to C infty B H(F)},
	\begin{equation}
		\label{E_1_sm_fib_eq}
		E_1^{p,q} = \Gamma\big(B,\Omega^p_B\otimes \mathcal{H}^q_{\mathrm{dR}}(F)\big) =: \Omega^p\big(B;\mathcal{H}^q_{\mathrm{dR}}(F)\big).
	\end{equation}
	The differential $d_1$ is the de~Rham differential $d_B$ twisted by the
	flat connection on the local system~$\mathcal H^q_{\mathrm{dR}}(F)$.
	In particular, if the monodromy is trivial (so that $\mathcal H^q_{\mathrm{dR}}(F)$
	is a constant sheaf), then $d_1 = d_B$.

	\subsection*{The $E_2$-page}
	By definition,
	\[
	E_2^{p,q}
	= H^p\big(E_1^{\bullet,q},d_1\big)
	= H^p_{\mathrm{dR}}\big(B;\mathcal H^q_{\mathrm{dR}}(F)\big),
	\]
	i.e. it is the de~Rham cohomology of $B$ with coefficients in the local system.
	In the case of trivial monodromy this identifies with the usual tensor product:
	\begin{equation}
		\label{E_2_sm_fib_eq}
		E_2^{p,q} = H^p_{\mathrm{dR}}\big(B;\underline{H^q_{\mathrm{dR}}(F)}\big) \cong H^p_{\mathrm{dR}}\big(B;H^q_{\mathrm{dR}}(F)\big) \cong H^p_{\mathrm{dR}}(B)\otimes H^q_{\mathrm{dR}}(F).
	\end{equation}

	\section{Relative disk mapping space and open-closed iterated integral map}
	\label{s_relative_disk_space}
	
	Let $M,N$ be smooth manifolds, and let $f:N\to M$ be a smooth map. Let $\mathbb D=\{z\in\C\mid\lvert z \rvert \leq 1\}$ denote the standard 2-disk. Define the\textit{ relative disk mapping space }
	\[
	\mathscr X:=
	\Map_f((\mathbb D,S^1),(M,N)):= \{(\Phi,\gamma)\mid \Phi:\mathbb D\to M,\ \gamma:S^1\to N,\ \Phi(e^{2\pi i t})=f (\gamma(t)) \ \forall t\in S^1 \}.
	\]
	where we identify $S^1$ with $[0,1]/\{0,1\}$.
	If $f$ is an embedding, then $\mathscr X$ is the space of disks in $M$ bounded by $N$; but, in general, $f$ need not be injective.
    \subsection{The differentiable space structure}
	We set up a differentiable space structure on $\mathscr X$ whose defining predifferentiable space structure is as follows.
	Let $\phi: U\to \mathscr X$ be a set map, and denote by $(\Phi^u, \gamma^u)$ the $\phi$-image of $u\in U$. We define such $\phi:U\to\mathscr X$ to be a plot if its \textit{uncurrying}
	\[
	\tilde \phi: U\times \mathbb D \times S^1 \to M\times N, \qquad \qquad (u,z,t)\mapsto (\Phi^u(z), \gamma^u(t))
	\]
	is a smooth map in the usual sense.
	
	
	Set $Q_{\ell,k}:=\mathbb D^\ell \times \Delta^k$, where $\mathbb D$ is the closed unit disk and $\Delta^k$ is the standard $k$-simplex:
    \[
    \mathbb D = \big\{z\in\mathbb C\mid |z|\leqslant 1\big\},\qquad \Delta^k= \big\{
	(t_1,\dots, t_k) \in [0,1]^k \mid 0\leqslant t_1\leqslant \cdots \leqslant t_k \leqslant 1
	\big\}.
    \]
	We claim that the evaluation map
	\begin{align*}
		\Ev_{\ell,k}:  Q_{\ell,k} \times \mathscr X &\to M^\ell \times N^{k+1} \\
		(z_1,\dots,z_\ell,t_1,\dots,t_k, \Phi,\gamma )&\mapsto (\Phi(z_1),\dots,\Phi(z_\ell),\gamma(0),\gamma(t_1),\dots,\gamma(t_k))
	\end{align*}
	is a differentiable map.
	To prove this, we first recall that the underlying predifferentiable space structure on $Q_{\ell,k} \times \mathscr X$ consists of plots of the type
	$\lambda\times \phi$
	where $\lambda:V\to Q_{\ell,k}$ is a smooth map from a convex set $V$ and $\phi:U\to\mathscr X$ is a plot of $\mathscr X$.
	Our goal is to show 
	\[
	\Ev_{\ell,k}\circ \phi:U\to M^\ell\times N^{k+1}
	\] 
	is also a plot, or equivalently a smooth map between manifolds. 
	First, by definition, the uncurrying 
	\[
	\tilde{\phi }:V\times \mathbb D\times S^1\to M\times N
	\]
	of $\phi$ is a smooth map where $\tilde \phi(u,z,t)=(\Phi^u(z),\gamma^u(t))$ is as described before.
	We also write $\lambda(v)=(z_1^v,\dots, z_\ell^v, t_1^v,\dots, t_k^v)$ for $v\in V$.
	Then, the smoothness of $\tilde \phi$ and $\lambda$ implies that $(v,u)\mapsto u\mapsto \Phi^u(z_i^u)$ and $(v,u) \mapsto v\mapsto  \gamma^v(t_j^v)$ are smooth maps. Hence,
	\[
	\Ev_{\ell,k}\circ(\lambda\times \phi)(v,u) = (\Phi^u(z_1^v),\dots, \Phi^u(z_\ell^v), \gamma^u(t_1^v),\dots, \gamma^u(t_k^v) )
	\]
	is also a smooth map in $(v,u)\in V\times U$.

	\subsection{Boundary components of $Q_{\ell,k}$}
	We are interested in the boundary $\partial Q_{\ell,k}$ of the manifold with corners
	$Q_{\ell,k}$ and its induced boundary orientation.
	
	On the one hand, the boundary of $\Delta^k$ is given by $k+1$ copies of $\Delta^{k-1}$. Specifically, we define $\underline s_j^k:\Delta^{k-1}\to \Delta^k$ by \[
	\underline s^k_j(\tau_1,\dots, \tau_{k-1}) =
	\begin{cases}
		(0,\tau_1,\dots, \tau_{k-1}) & \text{if} \ j=0 \\
		(\tau_1,\dots,\tau_{j-1} ,\tau_j,\tau_j,\tau_{j+1},\dots, \tau_{k-1}) & \text{if} \ 1\leqslant j\leqslant k-1 \\
		(\tau_1,\dots, \tau_{k-1},1) & \text{if} \ j=k.
	\end{cases}
	\]
	For $0\leqslant j\leqslant k$, we produce the embeddings 
	\begin{equation}
		\label{embedding_1_eq}
		s_j=s_j^{\ell,k}:=\id_{\mathbb D^\ell}\times \underline s_j^k : \  Q_{\ell,k-1}\to  Q_{\ell,k}.
	\end{equation}
	Besides, the orientation sign of $s_j=s_j^{\ell,k}$ is $(-1)^{j+1}$.

	On the other hand, identifying $S^1=\partial\mathbb D$ with $[0,1]/\{0,1\}$, we observe that
	\[
	S^1\times \Delta^k = \bigcup_{i=0}^k \{(t, (t_1,\dots, t_k)\in S^1\times \Delta^k\mid t_i\leqslant t\leqslant t_{i+1}\} =: \bigcup_{i=0}^k \Delta^k_{(i)}
	\]
	can be decomposed into $k+1$ copies of $(k+1)$-simplices,
	where we temporarily set $t_0=0,t_{k+1}=1$.
	Here $\Delta_{(i)}^k$ is a closed subset of $S^1\times \Delta^k$ identified with the standard $(k+1)$-simplex through the map
	\[
	\underline \sigma^k_{i}: \Delta^{k+1} \xrightarrow{\cong} \Delta^k_{(i)} \subseteq S^1\times \Delta^k , \qquad 
	(t_1,\dots,t_{k+1})
	\mapsto (t_{i+1}, (t_1,\dots,\widehat{t_{i+1}},\dots,t_{k+1})).
	\]
	Restricting the standard orientation of $S^1\times \Delta^k$ to $\Delta_{(i)}^k$, one can check that the orientation sign of $\underline \sigma^k_{i}$ is $(-1)^i$.
	The boundary of $\mathbb D^\ell$ is the disjoint union of
	$\mathbb D^{r-1}\times S^1\times \mathbb D^{\ell-r}
	$ for $1\leqslant r\leqslant \ell$. There is an orientation-preserving identification $\underline \tau_r^\ell: \mathbb D^{\ell-1} \times S^1 \to \mathbb D^{r-1}\times S^1\times \mathbb D^{\ell-r}$ given by $ ((z,z'),t)\mapsto (z,t,z')$.
	Therefore, for all $1\leqslant r\leqslant \ell$ and $0\leqslant i\leqslant k$, we also produce the embeddings
	\begin{equation}
		\label{embedding_2_eq}
		\iota_{r,i}=\iota_{r,i}^{\ell,k}: \ Q_{\ell-1,k+1}\xrightarrow{\id_{\mathbb D^{\ell-1}}\times \underline \sigma_i^k}  \mathbb D^{\ell-1}\times S^1\times \Delta^k \xrightarrow{\underline \tau_r^\ell \times \id_{\Delta^k}} \mathbb D^{r-1}\times S^1\times \mathbb D^{\ell-r}\times \Delta^k \subseteq Q_{\ell,k}
	\end{equation}
	whose orientation sign is given by $(-1)^{i}$.
	
	Finally, notice that the boundary $\partial Q_{\ell,k}$ is given by the disjoint union of $\mathbb D^{r-1}\times S^1\times \mathbb D^{\ell-r} \ (1\leqslant r\leqslant \ell)$ and $\mathbb D^\ell\times\partial \Delta^k$.
	Further considering (\ref{embedding_1_eq}) and (\ref{embedding_2_eq}), we have 
	\begin{equation}
		\label{orientation_sign_boundary}
		\int_{\partial Q_{\ell,k}}i_{Q_{\ell,k}}^* = 
		\sum_{j=0}^k (-1)^{j+1} \int_{Q_{\ell,k-1}} s_j^* + 
		\sum_{i=0}^k  \sum_{r=1}^{\ell}  (-1)^i \int_{Q_{\ell-1,k+1}} \iota_{r,i}^*
	\end{equation}
	where $i_{Q_{\ell,k}}: \partial Q_{\ell,k}\to Q_{\ell,k}$ is the natural immersion.

	\subsection{The cochain map $J$}

	Let
	$\pi^{\ell,k}_{M,i}:M^\ell \times N^{k+1}\to M$ and
	$\pi^{\ell,k}_{N,j}:M^\ell\times N^{k+1}\to N$
	be the projection maps onto the $i$-th copy of $M$ and the $j$-th copy of $N$, respectively. For simplicity, we introduce
	\begin{align*}
		\bigtimes_{i=1}^\ell \omega_i\times \bigtimes_{j=0}^k\eta_j\ &=\ 
		\omega_1\times \cdots \times \omega_\ell\times \eta_0\times \cdots\times \eta_k \\
		&= \ (\pi^{\ell,k}_{M,1})^*\omega_1 \wedge \cdots \wedge (\pi^{\ell,k}_{M,\ell})^*\omega_\ell \wedge (\pi^{\ell,k}_{N,0})^*\eta_0\wedge \cdots \wedge (\pi^{\ell,k}_{N,k})^*\eta_k.
	\end{align*}
Slightly abusing the notations, we also write $\omega_{[\ell]}$ for $\omega_1\times\cdots\times \omega_\ell$ and write $\eta_{[0,k]}$ for $\eta_0\times \cdots\times \eta_k$.

	Define a linear map 
	\begin{equation}\label{equation: J_l,k}
		J_{\ell,k} : \Omega(M)^{\otimes \ell}\otimes\Omega(N)^{\otimes k+1}\to\Omega^\bullet(M^\ell\times N^{k+1})
		\to\Omega^\bullet(Q_{\ell,k}\times \mathscr X) 
		\to\Omega^{\bullet-2\ell-k}(\mathscr X)
	\end{equation}
	by
	\[
	J_{\ell,k}(\omega_1\otimes\dots\otimes\omega_{\ell}\otimes \eta_{[0,k]}) = 
	(-1)^{\epsilon_{\ell,k}} \int_{Q_{\ell,k}} \Ev_{\ell,k}^* \left(\bigtimes_{i=1}^\ell \omega_i\times \bigtimes_{j=0}^k\eta_j\right)
	\]
	with sign
	\begin{equation}
		\label{sign_eq}
		\epsilon_{\ell, k}
		=\epsilon_{\ell, k}(\omega_{[\ell]}\otimes\eta_{[0,k]})
		=
		\ell+(k+1)|\omega_{[\ell]}|+
		\sum_{a=0}^k (k-a)|\eta_a|, 
	\end{equation}
	where we adopt the shifted degrees in (\ref{shifted_degree_eq}).
	Since
	$J_{\ell,k}$ is of internal degree $-2\ell-k$, that is, $\deg J_{\ell,k} (\omega_{[\ell]}\otimes \eta_{[0,k]})=\sum_i \deg \omega_i +\sum_i \deg \eta_i - (2\ell +k)$, we have $|J_{\ell,k}|=0$ for the shifted degree.
	
	\begin{lem}\label{lemma: J is symmetric}
		For any $l,k\geq0$, permutation $\sigma\in S_{\ell}$, $\omega_1,\dots,\omega_\ell\in\Omega(M)$ and $\eta_0,\dots,\eta_k\in\Omega(N)$, there holds
		\[
		J_{\ell,k}(\omega_{\sigma(1)}\otimes \cdots \otimes \omega_{\sigma(\ell)} \otimes \eta_{[0,k]}) = (-1)^{\epsilon(\sigma)} \ J_{\ell,k}(\omega_{1}\otimes \cdots \otimes \omega_{\ell} \otimes \eta_{[0,k]})
		\]
		where 
		$\epsilon(\sigma)= \sum_{i<j,\ \sigma(i)>\sigma(j)} |\omega_{\sigma(i)}||\omega_{\sigma(j)}|$.
		
	\end{lem}

	
	\begin{proof}
		
		Since $Q_{\ell,k}$ is a convex set in $\mathbb C^\ell\times \mathbb R^k$, $\lambda=\mathrm{id}_{Q_{\ell,k}}$ is a canonical convex parameterization of $Q_{\ell,k}$. Let $\phi:U\to\mathscr X$ be a plot of $\mathscr X$, and let $\xi=(\xi^1,\dots,\xi^p)$ be coordinates on $U$. Fix $l,k$, denote 
		\[
		e_{M,i}=\pi_{M,i}^{l,k}\circ\Ev_{l,k}:Q_{\ell,k}\times\mathscr X\to M\ (1\leq i\leq \ell), \quad e_N=(\pi_{N,j}^{l,k})_{0\leq j\leq k}:Q_{\ell,k}\times\mathscr X\to N^{k+1}.
		\]
		Then 
		\begin{align*}
			\Ev_{\ell,k}^* \left(\bigtimes_{i=1}^\ell \omega_{i}\times \bigtimes_{j=0}^k\eta_j\right)_{\lambda\times\phi} &=  (e_{M,1}^* \omega_{1})_{\lambda\times\phi}\wedge\dots\wedge (e_{M,\ell}^* \omega_{\ell})_{\lambda\times\phi}\wedge(e_N^*\eta_{[0,j]})_{\lambda\times\phi} \\
			&= \alpha_1(z_1,\xi)\wedge\dots\wedge\alpha_\ell(z_\ell,\xi)\wedge (e_N^*\eta_{[0,j]})_{\lambda\times\phi},
		\end{align*}
		where $\alpha_i(z_i,\xi)=(e_{M,i}^* \omega_{i})_{\lambda\times\phi}$ and $z_i=(x_i,y_i)$ is the coordinate on the $i$-th copy of $\mathbb D$ in $Q_{\ell,k}=\mathbb D^l\times\Delta^k$.
		Similarly,
		\begin{align*}
			\Ev_{\ell,k}^* \left(\bigtimes_{i=1}^\ell \omega_{\sigma(i)}\times \bigtimes_{j=0}^k\eta_j\right)_{\lambda\times\phi} &= (e_{M,1}^* \omega_{\sigma(1)})_{\lambda\times\phi}\wedge\dots\wedge(e_{M,\ell}^* \omega_{\sigma(\ell)})_{\lambda\times\phi}\wedge (e_N^*\eta_{[0,j]})_{\lambda\times\phi} \\
			& = \alpha_{\sigma(1)}(z_1,\xi)\wedge\dots\wedge\alpha_{\sigma(\ell)}(z_\ell,\xi)\wedge (e_N^*\eta_{[0,j]})_{\lambda\times\phi} \\
			& =(-1)^{\epsilon(\sigma)} \alpha_{1}(z_{\sigma^{-1}(1)},\xi)\wedge\dots\wedge\alpha_{\ell}(z_{\sigma^{-1}(\ell)},\xi)\wedge (e_N^*\eta_{[0,j]})_{\lambda\times\phi}.
		\end{align*}
		It remains to show
		\[    \int_{Q_{\ell,k}}\alpha_1(z_1,\xi)\wedge\dots\wedge\alpha_\ell(z_\ell,\xi)\wedge (e_N^*\eta_{[0,j]})_{\lambda\times\phi} =     \int_{Q_{\ell,k}}\alpha_{1}(z_{\sigma^{-1}(1)},\xi)\wedge\dots\wedge\alpha_{\ell}(z_{\sigma^{-1}(\ell)},\xi)\wedge (e_N^*\eta_{[0,j]})_{\lambda\times\phi}.
		\]
		Observe that both integrals could be nonzero only if every $\alpha_i(z_i,\xi)$ is in the form $a_i(z_i,\xi)\,dx^i\wedge dy^i\wedge d\xi^{J_i}$; in such case, since $\deg (dx^i\wedge dy^i)=2$, the result follows from Fubini's theorem.
	\end{proof}

	Clearly, $J_{\ell,k}$ vanishes on elements of the form 
	$\omega_1\otimes\dots\otimes\omega_\ell\otimes\eta_0\otimes\dots\otimes\eta_k$ whenever some $\omega_i$ is a $p$-form with $p>\dim \mathbb D = 2$, or some $\eta_j$ ($j>0$) is a $q$-form with $q>\dim S^1 = 1$; in particular, $J_{\ell,k}$ vanishes if some $\omega_i$ or some $\eta_j$ ($j>0$) is a unit. Combining this fact with Lemma \ref{lemma: J is symmetric}, we conclude that $J_{\ell,k}$ reduces to a linear map
	\[
	\overline{\Omega(M)}^{\wedge\ell} \otimes \Omega(N) \otimes \overline{\Omega(N)}^{\otimes k} \to \Omega(\mathscr X)
	\]
	which we still denote by $J_{\ell,k}$.
	We call $J=(J_{\ell,k})_{\ell,k\geq0}$ the \textit{open-closed iterated integral map} associated to $f:N\to M$.
	
	\begin{thm}[Theorem \ref{main_thm}]
		\label{J_cochain_map}
		$J=(J_{\ell,k})$ gives a cochain map from the open-closed Hochschild chain complex $C_\bullet(\Omega(M); \Omega(N))$ to the de Rham complex $\Omega^\bullet(\mathscr X)$.
	\end{thm}
	
The starting point of the proof is the Stokes formula for integration along fibers (Theorem \ref{integration_along_fiber_thm}), from which we initially conclude that
	\begin{align*}
		&(-1)^{2\ell+k\,}d_{\mathscr X} \int_{Q_{\ell,k}} \Ev_{\ell,k}^* (\omega_{[\ell]}\times \eta_{[0,k]})   \quad + \quad \int_{\partial Q_{\ell,k}} i_{Q_{\ell,k}}^* \Ev_{\ell,k}^* (\omega_{[\ell]}\times \eta_{[0,k]}) \\
		=\,
		&\sum_{i=1}^\ell (-1)^{|\omega_{[i-1]}|} \int_{Q_{\ell,k}} \Ev_{\ell,k}^* ( \omega_{[i-1]} \times d \omega_i \times \omega_{[i+1,\ell]}\times \eta_{[0,k]}) \\
		&+ \sum_{j=1}^k (-1)^{|\omega_{[\ell]}|+|\eta_{[0,j-1]}|+j} \int_{Q_{\ell,k}} \Ev_{\ell,k}^* ( \omega_{[\ell]} \times \eta_{[0,j-1]}\times d\eta_j\times \eta_{[j+1,k]})  \\
		&+ (-1)^{|\omega_{[\ell]}|} 
		\int_{Q_{\ell,k}} \Ev_{\ell,k}^* (\omega_{[\ell]} \times d\eta_0\times \eta_{[k]}),
	\end{align*}
	where we use the standard induced de Rham differential on $M^\ell\times N^{k+1}$.
	
	Equivalently, by (\ref{b_i012_eq}), multiplying it with the sign in (\ref{sign_eq}) yields
	\begin{equation}
		\label{applying_integration_along_fiber_1_eq}
		(-1)^kd_{\mathscr X} J_{\ell,k} +  \int_{\partial Q_{\ell,k}} i_{Q_{\ell,k}}^* \Ev_{\ell,k}^* \,  (-1)^{\epsilon_{\ell,k}} = 
		(-1)^k\left(
		J_{\ell,k}\circ b_0 + J_{\ell,k} \circ b_{1,1}
		+ J_{\ell,k} \circ b_{2,1} \right)
	\end{equation}
	where we may view $(-1)^{\epsilon_{\ell,k}}$ as an operator on $\Omega(M)^{\otimes \ell}\otimes \Omega(N)^{\otimes k}$:\[
	\omega_{[\ell]} \otimes \eta_{[0,k]} \mapsto
	(-1)^{\epsilon_{\ell,k}} \  \omega_{[\ell]}\otimes  \eta_{[0,k]}.
	\]
	
Based on (\ref{orientation_sign_boundary}), we obtain
	\begin{equation}
		\label{applying_integration_along_fiber_2_eq}
		\begin{aligned}
			(-1)^{\epsilon_{\ell,k}} \int_{\partial Q_{\ell,k}} i_{Q_{\ell,k}}^* \Ev_{\ell,k}^* (\omega_{[\ell]}\times \eta_{[0,k]})
			&=
			\sum_{j=0}^k (-1)^{j+1+\epsilon_{\ell,k}} \int_{Q_{\ell,k-1}} s_j^* \Ev_{\ell,k}^* ( \omega_{[\ell]}\times \eta_{[0,k]}) \\
			&+
			\sum_{i=0}^k \sum_{r=1}^{\ell} 
			(-1)^{i+\epsilon_{\ell,k}} \int_{Q_{\ell-1,k+1}}  \iota_{r,i}^* \Ev_{\ell,k}^* ( \omega_{[\ell]}\times \eta_{[0,k]}).
		\end{aligned}
	\end{equation}

Next, we need to further compute these terms,  with some particular care for sign computations.

\subsection{Proof of Theorem \ref{J_cochain_map}}
	We compute the first term on the right hand side of (\ref{applying_integration_along_fiber_2_eq}).
	By the definitions of $s_j$'s, one can check the following diagrams commute:
	\[
	\xymatrix{
		Q_{\ell,k-1} \times \mathscr X \ar[rrr]^{\Ev_{\ell,k-1}} \ar[dd]^{s_j:=s_j\times \id_{\mathscr X}} & & & M^\ell\times N^k \ar[dd]^{\Gamma_j} \\
		\\
		Q_{\ell,k} \times \mathscr X \ar[rrr]^{\Ev_{\ell,k}}  & & & M^\ell\times N^{k+1}
	}
	\]
	where 
	$\Gamma_j =
	\id_{M^\ell}\times \id_{N^{j}}\times \Delta_N \times \id_{N^{k-1-j}}$
	with the diagonal map $\Delta_N:N\to N\times N$ for $0\leqslant j\leqslant k-1$
	and where
	$\Gamma_n = \id_M^{\ell}\times \Gamma_n'$
	with $\Gamma_n'(n_0,n_1,\dots, n_k) = (n_0,n_1,\dots, n_k, n_0)$.
	We further consider the diagram
	\[
	\xymatrix{
		M^\ell\times N^k \ar[dd]^{\Gamma_j} \ar[rrd]^{\pi} \\ & & M \ \text{or} \ N \\
		M^\ell\times N^{k+1} \ar[rru]_{\pi'}
	}
	\]
	which, by the definition of $\Gamma_j$, commutes for the following cases (see also the table):
	\begin{itemize}
		\item 
		when $0\leqslant j\leqslant k-1$ and the pair $(\pi,\pi')$ is: 
		$(\pi_{M,a}^{\ell,k-1}, \pi_{M,a}^{\ell,k})$ for $1\leqslant a\leqslant \ell$;
		$(\pi_{N,a}^{\ell,k-1} , \pi_{N,a}^{\ell,k})$ for $0\leqslant a\leqslant j-1$; $(\pi_{N,j}^{\ell,k-1} , \pi_{N,j}^{\ell,k})$; $(\pi_{N,j}^{\ell,k-1} , \pi_{N,j+1}^{\ell,k})$; or $(\pi_{N,a}^{\ell,k-1} , \pi_{N,a+1}^{\ell,k})$ for $j+1\leqslant a\leqslant k-1$.
		\item 
		when $j=k$ and the pair $(\pi,\pi')$ is $(\pi_{M,a}^{\ell,k-1}, \pi_{M,a}^{\ell,k})$ for $1\leqslant a\leqslant \ell$; $(\pi_{N,0}^{\ell,k-1}, \pi_{N,k}^{\ell,k})$; $(\pi_{N,0}^{\ell,k-1}, \pi_{N,0}^{\ell,k})$; or $(\pi_{N,a}^{\ell,k-1},\pi_{N,a}^{\ell,k})$ for $1\leqslant a\leqslant k-1$.
	\end{itemize}
	\begin{table}[h]
		\centering
		\begin{tabular}{|l||lllll||llll|}
			\hline
			&
			\multicolumn{5}{c||}{$j=0,1,\dots, k-1$} &
			\multicolumn{4}{c|}{$j=k$} \\ 
			\hline \rule{0pt}{15pt}
			$\pi$ & 
			\multicolumn{1}{l|}{$\pi_{M,a}^{\ell,k-1}$} &
			\multicolumn{1}{l|}{$\pi_{N,a}^{\ell,k-1}$} &
			\multicolumn{2}{c|}{$\pi_{N,j}^{\ell,k-1}$} & 
			$\pi_{N,a}^{\ell,k-1}$
			&
			\multicolumn{1}{l|}{$\pi_{M,a}^{\ell,k-1}$} &
			\multicolumn{2}{c|}{$\pi_{N,0}^{\ell,k-1}$} & 
			$\pi_{N,a}^{\ell,k-1} $
			\\ 
			\hline \rule{0pt}{15pt}
			$\pi'$ &
			\multicolumn{1}{l|}{$\pi_{M,a}^{\ell,k}$} &
			\multicolumn{1}{l|}{$\pi_{N,a}^{\ell,k}$} &
			\multicolumn{1}{l|}{$\pi_{N,j}^{\ell,k}$} &
			\multicolumn{1}{l|}{$\pi_{N,j+1}^{\ell,k}$} &
			$\pi_{N,a+1}^{\ell,k}$
			&
			\multicolumn{1}{l|}{$\pi_{M,a}^{\ell,k}$} &
			\multicolumn{1}{l|}{$\pi_{N,k}^{\ell,k}$} &
			\multicolumn{1}{l|}{$\pi_{N,0}^{\ell,k}$} 
			&
			$\pi_{N,a}^{\ell,k}$
			\\ 
			\hline 
			& 
			\multicolumn{1}{l|}{$1\leqslant a\leqslant \ell$} &
			\multicolumn{1}{l|}{$0\leqslant a\leqslant j-1$} &
			\multicolumn{2}{c|}{} & 
			$j+1\leqslant a\leqslant k-1$
			&
			\multicolumn{1}{l|}{$1\leqslant a\leqslant \ell$} &
			\multicolumn{2}{c|}{} & 
			$1\leqslant a\leqslant k-1$
			\\ \hline
		\end{tabular}
	\end{table}
	For these choices of $(\pi,\pi')$, we have $\pi \circ \Ev_{\ell,k-1}=\pi'\circ \Ev_{\ell,k}\circ s_j$ and $\Ev_{\ell,k-1}^* \pi_Q^* = s_j^* \Ev_{\ell,k}^* (\pi')^*$ by joining the above two commutative diagrams.
	For instance, we may obtain
	\[
	\Ev_{\ell,k-1}^* (\pi_{N,0}^{\ell,k-1})^* (\eta\wedge \tilde \eta) = s_k^*\Ev_{\ell,k}^* \Big( (\pi_{N,k}^{\ell,k})^* \eta \wedge (\pi_{N,0}^{\ell,k})^* \tilde \eta \Big) .
	\]

	Putting things together, one can verify that 
	\begin{align*}
		& \sum_{j=0}^{k-1} \, (-1)^{j+1+\epsilon_{\ell,k}}  \int_{Q_{\ell,k-1}}s_j^*\Ev_{\ell,k}^* (\omega_{[\ell]}\times \eta_{[0,k]}) \\
		= \, &
		\sum_{j=0}^{k-1} \, (-1)^{j+1+\epsilon_{\ell,k}} \int_{Q_{\ell,k-1}}
		\Ev_{\ell,k-1}^* \left(\omega_{[\ell]}\times \eta_{[0,j-1]}\times \left(\eta_j\wedge \eta_{j+1}\right) \times \eta_{[j+2,k]}\right) \\
		= \, &  
		(-1)^{\dagger}  J_{\ell,k-1} \big(
		(b_{1,2}+b_{2,2}'')(\omega_{[\ell]}\otimes \eta_{[0,k]})
		\big) ,
	\end{align*}
	where the last sign is
	\begin{align*}
		\dagger
		= & \
		j+1+\ell+ (k+1)|\omega_{[\ell]}| +\sum_{a=0}^k(k-a)|\eta_a| \\
		& + 
		\ell+k|\omega_{[\ell]}|+\sum_{a=0}^{j-1}(k-1-a)|\eta_a| +(k-1-j)(|\eta_j|+|\eta_{j+1}|+1) +\sum_{a=j+2}^k(k-a)|\eta_a| \\
		& +
		|\omega_{[\ell]}|+|\eta_{[0,j]}|+1 \\
		= &\  k+1 .
	\end{align*}
	Similarly, for $j=k$, one has
	\begin{align*}
		&(-1)^{k+1+\epsilon_{\ell,k}} \int_{Q_{\ell,k-1}}s_k^*\Ev_{\ell,k}^* (\omega_{[\ell]}\times \eta_{[0,k]}) \\
		= \, &
		(-1)^{\dagger'} \ J_{\ell,k-1} \big(
		b_{2,2}' (\omega_{[\ell]}\otimes \eta_{[0,k]}) \big) ,
	\end{align*}
	where the sign is
	\begin{align*}
		\dagger' = \ & k+1+\ell+(k+1)|\omega_{[\ell]}|+\sum_{a=0}^k (k-a)|\eta_a| \\
		&+ \ell+k|\omega_{[\ell]}| + (k-1)(|\eta_k|+|\eta_0|+1) + \sum_{a=1}^{k-1} (k-1-a)|\eta_a| \\
		&+
		(|\eta_k|+1)(|\eta_{[0,k-1]}|+k) + |\eta_k||\eta_{[0,k-1]}|+|\omega_{[\ell]}|+|\eta_k|-1 \\
		=\ & k+1 .
	\end{align*}
	In summary,
	\begin{equation}
		\label{commutative_diagram_result_1_eq}
		\sum_{j=0}^k (-1)^{j+1+\epsilon_{\ell,k}} \int_{Q_{\ell,k-1}} s_j^* \Ev_{\ell,k}^* ( \omega_{[\ell]}\times \eta_{[0,k]}) 
		=
		(-1)^{k+1} \, J_{\ell,k-1}\circ (b_{1,2} +  b_{2,2} ) .
	\end{equation}

	Next, we compute the second term on the right hand side of (\ref{applying_integration_along_fiber_2_eq}).
	Using the definitions of $\iota_{r,i}$'s, we have the commutative diagrams
	\[
	\xymatrix{
		Q_{\ell-1,k+1} \times \mathscr X \ar[rrr]^{\Ev_{\ell-1,k+1}} \ar[dd]^{\iota_{r,i}:=\iota_{r,i}\times \id_{\mathscr X}} & & & M^{\ell-1}\times N^{k+2} \ar[dd]^{\Theta_{r,i}} \\
		\\
		Q_{\ell,k} \times \mathscr X \ar[rrr]^{\Ev_{\ell,k}}  & & & M^\ell\times N^{k+1}
	}
	\]
	for all $0\leqslant i\leqslant k$ and $1\leqslant r\leqslant \ell$. Here the map $\Theta_{r,i}$ is defined by
	\begin{align*}
	    (m_1,\dots, m_{\ell-1}, n_0,n_1&,\dots, n_i,n_{i+1},n_{i+2},\dots, n_{k+1}) \\ & \mapsto (m_1,\dots, m_{r-1}, f(n_{i+1}), m_{r},\dots, m_{\ell-1}, n_0, n_1,\dots, n_i,n_{i+2},\dots, n_{k+1}).
	\end{align*}	
	Moreover, the diagram
	\[
	\xymatrix{
		M^{\ell-1} \times N^{k+2}\ar[dd]^{\Theta_{r,i}} \ar[rrd]^{\pi} \\ & & M \ \text{or} \ N \\
		M^\ell\times N^{k+1} \ar[rru]_{\pi'}
	}
	\]
	commutes when the pair $(\pi,\pi')$ is given by a column in the following table:
	\begin{table}[h]
		\centering
		\begin{tabular}{|l|l|l|l|l|l|}
			\hline\rule{0pt}{15pt}
			$\pi$  & $\pi_{M,a}^{\ell-1,k+1}$ & $f\circ \pi_{N,i+1}^{\ell-1,k+1}$  & $\pi_{M,a-1}^{\ell-1,k+1}$ & $\pi_{N,a}^{\ell-1,k+1}$ & $\pi_{N,a+1}^{\ell-1,k+1}$ \\ 
			\hline\rule{0pt}{15pt}
			$\pi'$ & $\pi_{M,a}^{\ell,k}$ & $\pi_{M,r}^{\ell,k}$ & $\pi_{M,a}^{\ell,k}$ & $\pi_{N,a}^{\ell,k}$ & $\pi_{N,a}^{\ell,k}$  \\ \hline
			& $1\leqslant a\leqslant r-1$ &  & $r+1\leqslant a\leqslant \ell$ & $0\leqslant a\leqslant i$ & $i+1\leqslant a\leqslant k+1$ \\ \hline
		\end{tabular}
	\end{table}
	
	In the second column, we use the definition of $\mathscr X$ to see that if $(\Phi,\gamma)\in\mathscr X$, then $\Phi(e^{2\pi \sqrt{-1} t})=f(\gamma(t))$.
	By combining the above two diagrams, we obtain $\pi\circ \Ev_{\ell-1,k+1}=\pi'\circ \Ev_{\ell,k}\circ \iota_{r,i}$.
	Applying all these relations, one can verify that
	\begin{align*}
		&(-1)^{i+\epsilon_{\ell,k}}
		\int_{Q_{\ell-1,k+1}}\iota_{r,i}^* \Ev_{\ell,k}^* (\omega_{[\ell]}\times \eta_{[0,k]}) \\
		=& \
		(-1)^{\ddagger''} \ 
		\int_{Q_{\ell-1,k+1}} \Ev_{\ell-1,k+1}^* \left( 
		\omega_{[\ell]\setminus\{r\}} \times \eta_{[0,i]}\times f^*\omega_r \times \eta_{[i+1,k]}
		\right) \\
		=& \
		(-1)^{\ddagger'} \ J_{\ell-1,k+1} (\omega_{[\ell]\setminus\{r\}} \otimes \eta_{[0,i]}\otimes f^*\omega_r\otimes \eta_{[i+1,k]}) \\
		=& \ 
		(-1)^{\ddagger} \ J_{\ell-1,k+1} \big(b_{1,0} (\omega_{[\ell]}\otimes \eta_{[0,k]})
		\big)
	\end{align*}
	where the sign is
	\begin{align*}
		\ddagger
		=\ &
		i+\ell+(k+1)|\omega_{[\ell]}|+\sum_{a=0}^k(k-a)|\eta_a| \\
		&+|\omega_r|(|\omega_{[r+1,\ell]}|+|\eta_{[0,i]}|+i+1) \\
		& 
		+ (\ell-1) + (k+2)|\omega_{[\ell]\setminus\{r\}}|+\sum_{a=0}^{i}(k+1-a)|\eta_a| + (k+1-i-1)(1+|\omega_r|) +\sum_{a={i+1}}^{k} (k+1-a-a)|\eta_a| \\
		&+
		|\omega_{[\ell]\setminus\{r\}}| + |\eta_{[0,i]}|+ |\omega_r| (|\omega_{[r+1,\ell]}|+|\eta_{[0,i]}|) = k+1 .
	\end{align*}
	Thus, we obtain
	\begin{equation}
		\label{commutative_diagram_result_2_eq}
		\sum_{i=0}^k \sum_{r=1}^{\ell} 
		(-1)^{i+\epsilon_{\ell,k}} \int_{Q_{\ell-1,k+1}}  \iota_{r,i}^* \Ev_{\ell,k}^* 
		=
		(-1)^{k+1} \ J_{\ell-1,k+1} \circ b_{1,0} . 
	\end{equation}

	According to (\ref{applying_integration_along_fiber_1_eq}), (\ref{applying_integration_along_fiber_2_eq}), (\ref{commutative_diagram_result_1_eq}), and (\ref{commutative_diagram_result_2_eq}), we eventually obtain
	\[
	d_{\mathscr X}\circ J_{\ell,k} = J_{\ell,k-1}\circ (b_{1,2}+ b_{2,2})+ J_{\ell-1,k+1}\circ b_{1,0} + J_{\ell,k}\circ (b_0 +b_{1,1}+b_{2,1}),
	\]
	or equivalently,
	\[
	d_{\mathscr X}\circ J = J\circ b.
	\]
This completes the proof of Theorem \ref{J_cochain_map}.	
	
	\section{Quasi-isomorphism results for $J$}\label{section: quasi-isom results}

After establishing Theorem \ref{J_cochain_map}, the purpose of this section is to establish the following main result of this paper:
	
\begin{thm}[Theorem \ref{main_thm_2}]
	\label{thm: J quasi isomorphism}
		For any 2-connected $M$ which is contractible or rationally homotopy equivalent to an odd-dimensional sphere, and any simply connected $N$ of finite type, the iterated integral map
		\[
		J: C_\bullet(\Omega(M);\Omega(N)) \to \Omega^\bullet(\mathscr X)=\Omega^\bullet\big(\Map_f((\mathbb D,S^1),(M,N))\big)
		\]
		is a quasi-isomorhpism.
	\end{thm}

	To motivate the idea of proof, let's begin with some special cases.

	$\bullet$ If $M$ is a point, then $f:N\to M=\pt$ is the unique map and 
	\[
	\mathscr X = \Map_f((\mathbb D,S^1),(\pt,N))=\Map(S^1,N)=LN
	\]
	is the free loop space of $N$. Since 
	\[
	\overline{\Omega(\pt)}=\Omega(\pt)/\mathbb R = \mathbb R / \mathbb R = 0 ,
	\]
	we have 
	\begin{align*}
	C(\Omega(\pt);\Omega(N)) & =\bigoplus_{\ell\geq0,\, k\geq0} \bigg(\overline{\Omega(\pt)[2]}^{\wedge\ell} \otimes \Omega(N)[1] \otimes \overline{\Omega(N)[1]}^{\otimes k}\bigg)[-1] \\
    & = \bigoplus_{k\geq 0} \Omega(N) \otimes \overline{\Omega(N)[1]}^{\otimes k}  = C(\Omega(N)),
	\end{align*}
	which is the usual Hochschild chains of the dg algebra $\Omega(N)$; the Hochschild differentials also coincide. Moreover, the iterated integral map $J$ coincides with Chen's original map
	\[
	C(\Omega(N))\to\Omega(LN),
	\]
	which is known to be a quasi-isomorphism provided that $N$ is simply connected and of finite type.
	
	$\bullet$	If $N$ is a point and we denote by $p_M\in M$ the image of $
	f:N=\pt\to M$,
	then we have
	\[
	\Map_f((\mathbb D,S^1),(M,\pt)) = \Map((\mathbb D,S^1),(M,p_M)) = \Map_*(S^2,M),
	\]
	which is the based sphere space of the pointed space $(M,p_M)$. We have
	\[
	C(\Omega(M);\Omega(\pt)) = \bigoplus_{\ell\geq0,\, k\geq0} \overline{\Omega(M)}^{\wedge\ell} \otimes \Omega(\pt) \otimes \overline{\Omega(\pt)}^{\otimes k} = \bigoplus_{\ell\geq 0} \overline{\Omega(M)}^{\wedge \ell}
	\]
	with differential $b=b_0+b_1+b_2$ where $b_1=0$, $b_2=0$ and $b_0$ is induced by $d_M$. 
	
	\subsection{The case $N$ is a point} When $N$ is a point, the previous cochain map $J$ reduces to
	\begin{equation} \label{equation: J N=pt}
		J: \left(\bigoplus_{\ell\geq 0} \overline{\Omega(M)}^{\wedge \ell},d_M\right) \to \left(\Omega(\Map_*(S^2,M)),d\right) .
	\end{equation}
	One naturally asks for which manifolds $M$ the map \eqref{equation: J N=pt} is a quasi-isomorphism.
	A simple example with a positive answer is when $M$ is a point or, more generally, any contractible manifold. Beyond these cases, we have the following result.
	
	\begin{prop}
		\label{odd_dim_sphere_prop}
		If $M$ is 2-connected and is rationally homotopy equivalent to an odd-dimensional sphere $S^{2n+1}$ ($n\ge1$), then the map $J$ in \eqref{equation: J N=pt} is a quasi-isomorphism.
	\end{prop}
    
    For the proof, we need to invoke the classical result of Getzler-Jones \cite{getzler1994operads} concerning an iterated integral model for double based loop spaces. The relevant definitions and results are collected in Appendix~\ref{section:Getzler Jones double loop space}, which the reader may wish to consult before reading the proof.

	\begin{proof}[Proof of Proposition \ref{odd_dim_sphere_prop}]
	Consider the chain map
	\[
	T: \left(\bigoplus_{\ell\geq0}\overline{\Omega(M)}^{\wedge\ell},d_M\right) \to \left(\bigoplus_{\ell\geq0} e_2(\ell)\otimes_{S_\ell}\Omega(M,p_M)^{\otimes\ell},d_M+\delta\right)
	\]
	introduced in \eqref{equation: OCHA to Getzler-Jones}.
	By Theorem \ref{thm: Getzler Jones double loop space} and \eqref{equation: J compatible with rho}, to show $J$ is a quasi-isomorphism, it suffices to show $T$ is a quasi-isomorphism. Consider the increasing filtration $F^p$ ($p\geq0$) on the two complexes given by the truncation $F^p=\bigoplus_{\ell\leq p}$. Then $F^p$ is bounded below and exhaustive on both complexes, so the resulting spectral sequences converge by \cite[Theorem 5.5.1]{weibel1994introduction}.
	Note that $T$ preserves the filtration, thus induces a morphism of spectral sequences $\{T_{E_r}\}$ on pages $\{E_r\}$. Since $\delta$ does not survive on the $E_0$ page, the differential on $E_0$ contains only $d_M$. Recall that over a field of characteristic 0, the functor of (co)invariants for a finite group is exact. So for the $E_1$ page, we have
	\[
	T_{E_1}: \left(\bigoplus_{\ell\geq0} \overline{H(M)}^{\wedge\ell},0\right) \to \left(\bigoplus_{\ell\geq0} e_2(\ell)\otimes_{S_\ell} \overline{H(M)} ,\  \delta \right).
	\]
	We claim that $T_{E_1}$ is an isomorphism of $E_1$ pages, so that the standard comparison theorem \cite[Theorem 5.2.12]{weibel1994introduction} implies $T$ is a quasi-isomorphism.
	
	By assumption, $H^\bullet(M)\cong H^\bullet(S^{2n+1})\cong\Lambda(x)$ is the exterior algebra freely generated by an element $x$ at degree $2n+1$. Therefore, the cup product on $\overline{H(M)}$ vanishes, and consequently, $\delta = 0$. It remains to show $T_{E_1}$ is a linear isomorphsm. For $\ell=0,1$, $e_2(\ell)=\mathbb R$, so $T_{E_1}$ identifies the two sides for $\ell=1,2$. Since $\deg x$ is odd, $x\wedge x = - x\wedge x = 0$, so $\overline{H(M)}^{\wedge \ell }= \overline{\Lambda(x)}^{\wedge\ell}=0$ for all $\ell\geq2$. It remains to show
	\[        e_2(\ell)\otimes_{S_\ell}\overline{\Lambda(x)}^{\otimes\ell} = 0,\quad \forall \ell\geq2.
	\]
	Since $\deg x$ is odd, the 1-dimensional $S_\ell$-representation $\overline{\Lambda(x)}^{\otimes \ell}$ is isomorphic to the sign character of $S_\ell$. By a result of Lehrer-Solomon \cite[Proposition 4.7]{LEHRER1986410}, the $S_\ell$-representation $e_2(\ell)$ contains no copy of the sign character in its irreducible decomposition. Therefore, $e_2(\ell)\otimes \overline{\Lambda(x)}^{\otimes \ell}$ does not contain the trivial representation, or equivalently, there is no nontrivial $S_\ell$-invariant elements in $e_2(\ell)\otimes \overline{\Lambda(x)}^{\otimes \ell}$. This completes the proof.
\end{proof}

	\subsection{The case $N$ is simply connected}	
	Next, we aim to prove the following result that promotes the quasi-isomorphism property of $J$ from the special case $N=\pt$ to much more general manifolds $N$.
	
	\begin{prop}
		\label{N_pt_to_general_N_prop}
		Let $M$ be a 2-connected smooth manifold of finite type. If 
		\[
		J: C(\Omega(M);\Omega(N)) \to \Omega(\mathscr X)=\Omega\big(\Map_f((\mathbb D,S^1),(M,N))\big)
		\]
		is a quasi-isomorphism when $N=\pt$, then $J$ is a quasi-isomorphism for any simply connected $N$ of finite type.
	\end{prop}
	
	We prove Proposition \ref{N_pt_to_general_N_prop} by comparing the Serre spectral sequence from the fibration of certain mapping spaces and a similar spectral sequence on the open-closed Hochschild side. The idea goes back to Adams \cite{adams1956}, and our presentation is similar to Getzler-Jones-Petrack \cite{getzler1991differential}.
	
	We set 
	\[
	\mathscr X_0 :=\Map_f((\mathbb D,S^1),(M,\pt)) \cong \Map_*(S^2,M).
	\]
	The starting point is the following fibration
	\[
	\xymatrix{
		\mathscr X_0 \, \,  \ar@{^{(}->}[r] &  \mathscr X \ar[d]^{R} & (\Phi, \gamma) \ar@{|->}[d]\\
		& LN & \gamma 
	}
	\]
	where $R$ is the forgetful map sending $(\Phi,\gamma)$ to $ \gamma$. 
	Let $\mathrm{Map}(\mathbb D, M)$ be the space of smooth maps $\mathbb D\to M$.
	Observe that the map $R:\mathscr X\to LN$ is actually a pullback of the map $R_M$
	defined by restricting a disk map to the boundary:
	\[
	R_M: \mathrm{Map}(\mathbb D, M) \to LM \quad , \qquad \Phi \mapsto \Phi|_{\partial \mathbb D}.
	\]
	If we write $R'$ for the forgetful map $(\Phi,\gamma)\mapsto \Phi$, the pullback diagram is as follows:
	\begin{equation}
		\label{pullback_X_eq}
		{\xymatrix{
				\mathscr X\ar[d]^R \ar[rr]^{R'} & & \mathrm{Map} (\mathbb D, M) \ar[d]^{R_M}  \\
				LN \ar[rr]^{Lf} & & LM.
		}}
	\end{equation}
    
	It is not hard to verify the above maps $R$ and $R_M$ are smooth fibrations (as defined at the start of \S \ref{s_smooth_fibration}). Indeed, given $H_0: Y\to \Map(\mathbb D,M)$ and $h:[0,1]\times Y\to LM$ with $R_M\circ H_0=h(0,\cdot)$, we can construct $H:[0,1]\times Y\to \Map(\mathbb D,M)$ by shrinking each disk map $H_0(y)$ toward the center while gradually attaching along the outer annulus the family of boundary loops given by $h(s,y)$, $s\in[0,1]$. More precisely, $H$ can be taken as a smooth perturbation of the following continuous lift:
	\[
	\widetilde{H}(s,y)(z)=\begin{cases}
		H_0(y)(\frac{z}{1-\frac{s}{2}}) & \text{if } |z|\leq 1-\frac{s}{2}, \\
		h(2|z|-2+s,y)(\frac{z}{|z|}) & \text{if } 1-\frac{s}{2} \leq |z| \leq 1.
	\end{cases}
	\]	
	
	Since $M$ is 2-connected,  we have $\pi_1(LM)=0$, so the fibration $R_M:\Map(\mathbb D,M)\to LM$ has trivial monodromy. Hence, the pullback fibration $R:\mathscr X\to LN$ also has trivial monodromy.

	\subsection{Filtration on the de Rham complex of $\mathscr X$}
	
	There is a filtration $\mathcal F^p \Omega(\mathscr X)$ on
	$
	\Omega(\mathscr X)
	$
	by setting
	$
	\mathcal F^p\Omega(\mathscr X) $
	to be the ideal $R^*(\Omega^{p}(LN)) \wedge \Omega^{\bullet -p }(\mathscr X)$ just as in Section \ref{s_smooth_fibration}.
	Replacing the general smooth fibration $\pi:E\to B$ in Section \ref{s_smooth_fibration} with the specific smooth fibration $R:\mathscr X\to LN$, we can analyze the spectral sequence $\{E_r^{p,q}, d_r^{p,q}\}$ associated to the filtration $\mathcal{F}^p\Omega(\mathscr X)$:
	
	The $E_0$ page is as in (\ref{E_0_sm_fib_eq}) given by:
	\begin{align*}
		E_0^{p,q}  \cong \Omega^p(LN) \otimes_{C^\infty(LN)} \Omega^q(\mathscr X/LN)
	\end{align*}
	with the differential $d_0=1\otimes d_{\Omega(\mathscr X / LN)}  :E_0^{p,q}\to E_0^{p,q+1}$.
	
	The $E_1$ page is as in (\ref{E_1_sm_fib_eq}) given by
	\begin{align*}
		E_1^{p,q} \ \cong \Omega^p(LN) \otimes_{C^\infty(LN)} H^q_{\mathrm{dR}}(\mathscr X /  LN) 
		= 
		\Omega^p \big(LN; \mathcal H_{\mathrm{dR}}^q(\mathscr X_0)\big).
	\end{align*}
	
	Since the monodromy of the fibration $R$ is trivial, we conclude that the $E_2$ page is given as in (\ref{E_2_sm_fib_eq}) :
	\[
	E_2^{p,q} \cong H^p_{\mathrm{dR}}(LN)\otimes H^q_{\mathrm{dR}}(\mathscr X_0).
	\]

	\subsection{Filtration on the open-closed Hochschild complex}
	
	Consider the open-closed Hochschild chain complex 
	\[
	C^{\mathrm{pre}}:=C(\Omega(M);\Omega(N))=\bigoplus_{\ell,k} \bigg(\overline{\Omega(M)[2]}^{\wedge\ell} \otimes \Omega(N)[1] \otimes \overline{\Omega(N)[1]}^{\otimes k}\bigg)[-1]
	\]
	with the open-closed Hochschild differential $b$, as in Section \ref{s_manifold_ocha}.
	We choose sub-dgas $\mathcal A(M)\subseteq \Omega(M)$ and $\mathcal A(N)\subseteq \Omega(N)$ as in Lemma \ref{lemma: sub dga (M,N)}.
	Then, we similarly have the open-closed Hochschild chain complex
	\[
	C=C_\bullet = C(\mathcal A(M);\mathcal A(N))=\bigoplus_{\ell,k} \bigg(\overline{\mathcal A(M)[2]}^{\wedge\ell} \otimes \mathcal A(N)[1] \otimes \overline{\mathcal A(N)[1]}^{\otimes k}\bigg)[-1].
	\]
	Thanks to Lemma \ref{lemma: sub dga (M,N)}, we may replace $C^{\mathrm{pre}}$ by $C$ without affecting the open-closed Hochschild homology.
	
	Consider the decreasing filtration
	\[
	\mathcal{F}'^p C = \langle \omega_{[\ell]}\otimes\eta_{[0,k]}: |\eta_{[0,k]}|+1 \geq p \rangle, \qquad p\geq0.
	\]
	One can verify that the Hochschild differential $b$ preserves the filtration. Indeed, for the decomposition $b_{i,j}$ at (\ref{b_i012_eq}), we have
	\begin{align*}
		b_0 \bigl(\mathcal {F}'^p C_\bullet \bigr) & \subseteq  \mathcal {F}'^p C_{\bullet+1} \\
		(b_{1,1}+b_{1,2}+b_{2,1}+b_{2,2})\bigl(\mathcal {F}'^p C_\bullet \bigr)  & \subseteq \mathcal {F}'^{p+1} C_{\bullet+1} \\
		b_{1,0}\bigl(\mathcal {F}'^p C_\bullet \bigr)  & \subseteq \mathcal {F}'^{p+2} C_{\bullet+1}.
	\end{align*}
	Here the last equality uses the fact that
	$\overline{\mathcal A(M)}$ contains only nonzero elements of de Rham degree $\ge 3$.
	We also remark that $b_0$ is induced by $d_M$; $b_{1,1}$ and $b_{2,1}$ are induced by $d_N$; $b_{1,0}$ is induced by $f^*$; and $b_{1,2}$ and $b_{2,2}$ are induced by the wedge product on $N$.
	
	We claim that the decreasing filtration $\mathcal{F}'^p$ ($p\geq0$) is bounded. Indeed, for any $n$, $\mathcal{F}^pC_n$ is spanned by the elements $\omega_{[\ell]}\otimes\eta_{[0,k]}$ with 
	\[
	|\omega_{[\ell]}|+|\eta_{[0,k]}|+1=n, \quad |\eta_{[0,k]}|+1\geq p.
	\]
	For $\mathcal{F}^pC_n\neq0$, one must have $|\omega_{[\ell]}|=\sum_j(\deg\omega_j-2)\geq0$ since $\overline{\mathcal A^i(M)}=0$ ($i\le2$), implying $p\leq n$. In other words $\mathcal{F}'^pC_n=0$ for all $p>n$, so $\mathcal{F}'$ is bounded. Consequently, the resulting spectral $\{E_r'^{p,q},d_r'^{p,q}\}$ converges to $H(C(\mathcal{A}(M);\mathcal{A}(N)),b)$ by \cite[Theorem 5.5.1]{weibel1994introduction}.

	The $E_0'$ page is
	\[
	E_0'^{p,q} \cong \langle \omega_{[\ell]}\otimes\eta_{[0,k]}: |\eta_{[0,k]}|+1 = p \ , \ |\omega_{[\ell]}|+|\eta_{[0,k]}|+1 = p+q \rangle .
	\]
	As only $b_0$ survives in $d_0'^{p,q}:E_0'^{p,q} \to E_0'^{p,q+1}$, we obtain
	\begin{align*}
		E_1'^{p,q} &= 
		\langle
		\omega_{[\ell]}\otimes \eta_{[0,k]} \in \overline{H(M)[2]}^{\wedge \ell} \otimes \mathcal A(N)[1]\otimes \overline{ \mathcal A(N)[1]}^{\otimes k} \ : \ |\omega_{[\ell]}|=q  \ , \ |\eta_{[0,k]}| +1 =p 
		\rangle 
	\end{align*}
	with the differential
	$
	d_1'^{p,q}:E_1'^{p,q}\to E_1'^{p+1,q}
	$.
	Since $b_{1,0}$ (the only part induced by $f^*$) does not survive in $d_1'$, the differential $d_1'$ acts like the usual Hochschild differential, and hence
	\[
	E_2'^{p,q} = 
	\langle
	\omega_{[\ell]}\otimes \eta_{[0,k]} \in \overline{H(M)[2]}^{\wedge \ell} \otimes HH(\mathcal A(N)) \ : \ |\omega_{[\ell]}|=q  \ , \ |\eta_{[0,k]}| +1 =p 
	\rangle 
	\]
	where $HH(\mathcal A(N))$ is the usual (reduced) Hochschild homology of the dga $\mathcal A(N)$.

	\subsection{Comparison between the two spectral sequences}
	
	Consider the restriction of the iterated integral map $J$, which we still denote by $J$:
	\[
	J: C(\mathcal A(M); \mathcal A(N)) \xhookrightarrow{\simeq} \ C(\Omega(M);\Omega(N)) \to \mathscr X .
	\]
	
	\begin{lem} 
		\label{J_preserve_filtration_lem}
		The map $J$ preserves the filtration, that is, 
		\[
		J\Big(  \mathcal F'^p C(\mathcal A(M); \mathcal A(N))
		\Big) \subseteq \mathcal F^p \Omega(\mathscr X), \qquad  p\geq0.
		\]
	\end{lem}
	
	\begin{proof}
		Recall that the map $J=(J_{\ell,k})$ is defined by $J_{\ell,k}(\omega_{[\ell]}\otimes \eta_{[0,k]} ) = \int_{Q_{\ell,k}} \Ev_{\ell,k}^* (\omega_{[0,k]}\times \eta_{[0,k]})$
		where $Q_{\ell,k}=\mathbb D^\ell\times \Delta^k$; see \eqref{equation: J_l,k}.
		By the diagram (\ref{pullback_X_eq}), we observe that the evaluation map $\Ev_{\ell,k}$ can be decomposed as follows:
		\[
		\xymatrix{
			\mathbb D^{\ell} \times \Delta^k \times \mathscr X \ar[rrr]^{\Ev_{\ell,k}}  \ar[d]_{\id\times R' \times R} & & & M^\ell \times N^{k+1} \\
			\mathbb D^\ell \times \Delta^k \times \Map(\mathbb D, M) \times LN \ar[rrr]^{\cong} & & & \mathbb D^\ell  \times \Map(\mathbb D, M) \times \Delta^k \times LN 
			\ar[u]_{\Ev^M_\ell \times \Ev^N_k }
		}
		\]
		where 
		\[
		\Ev^M_\ell (z_1,\dots, z_\ell, \Phi) = (\Phi(z_1),\dots, \Phi(z_\ell))
		\]
		and 
		\[
		\Ev^N_k (t_0, t_1,\dots, t_k , \gamma) = (\gamma(t_0),\gamma(t_1),\dots, \gamma(t_k) )
		\]
		are the evident evaluation maps.
		Then,
		\[
		J_{\ell,k} (\omega_{[\ell]}\otimes \eta_{[0,k]}) =\pm (R')^*\left(\int_{\mathbb D^\ell} \Big(R^*\int_{\Delta^k}
		(\Ev_k^N)^*(\eta_{[0,k]})\Big) \wedge  (\Ev_\ell^M)^*(\omega_{[\ell]})\right).
		\]
		It follows that if $|\eta_{[0,k]}|+1 = \sum_i\deg\eta_i - k \geq p$, then $J_{\ell,k} (\omega_{[\ell]}\otimes \eta_{[0,k]})\in R^*\Omega^p(LN)\wedge\Omega(\mathscr X)$.
	\end{proof}
	
	\begin{proof}[Proof of Proposition \ref{N_pt_to_general_N_prop}]
		By Lemma \ref{J_preserve_filtration_lem}, $J$ induces a morphism of spectral sequences
		\[
		\{J_r\}:\{E_r'^{p,q},d_r'\} \to \{E_r^{p,q},d_r\}.
		\]
		On the $E_2$-page, it is in the form
		\[
		J_2=J_M\otimes J_N \ : \
		E_2'^{p,q} = H^q(C(\mathcal{A}(M);\mathcal{A}(\pt)))\otimes HH^p(\mathcal{A}(N)) \xrightarrow{} H^q_{\mathrm{dR}}(\mathscr X_0)\otimes H^p_{\mathrm{dR}}(LN) = E_2^{p,q}.
		\]
		By assumption, $J_M:H^q(C(\mathcal{A}(M);\mathcal{A}(\pt)))\to H^q_{\mathrm{dR}}(\mathscr X_0)$ is an isomorphism. Since $N$ is simply connected, by Chen's theorem \cite{chen1978pullback}, $J_N:HH^p(\mathcal{A}(N))\to H^p_{\mathrm{dR}}(LN)$ is an isomorphism. Therefore, $J_2$ is an isomorphism on the $E_2$ pages. Since both spectral sequences converge, by the standard comparison theorem (\cite[Theorem 5.2.12]{weibel1994introduction}), $J: C(\mathcal{A}(M);\mathcal{A}(N))\to \Omega(\mathscr X)$ is a quasi-isomorphism.
	\end{proof}
	
	Theorem \ref{thm: J quasi isomorphism} follows immediately from Proposition \ref{odd_dim_sphere_prop} and Proposition \ref{N_pt_to_general_N_prop}.

	\appendix
	\section{Getzler-Jones' construction for double loop spaces}\label{section:Getzler Jones double loop space}
	For any compact (more generally, finite type) smooth manifold $M$, Getzler-Jones \cite{getzler1994operads} constructed a cochain complex $B_2(\Omega(M))$ and a cochain map
	\[
	\rho: B_2(\Omega(M)) \to \Omega(\Map_*(S^2,M))
	\]
	which they prove to be a quasi-ismorphism if $M$ is 2-connected. The definition of $B_2(\Omega(M))$ involves configuration space of points in the complex plane (see \eqref{B_2_eq} below), and $\rho$ is defined via iterated integrals. The based sphere space $\Map_*(S^2,M)\cong \Map((\mathbb D,S^1),(M,p_M))$ is naturally identified with the double based loop space of $M$.
	
	In this section, we review Getzler-Jones' construction in a way compatible with our construction for the relative disk mapping space.
	
	\subsection{The cohomology of configuration spaces}
	Consider the configuration space of ordered $\ell$ distinct points in a topological space $X$, denoted by
	\[
	\mathrm{Conf}_\ell(X) = \big\{(x_1,\dots,x_\ell)\in X^\ell \mid x_i\neq x_j\ (\forall i\ne j)\big\}.
	\]
	For $X = \C = \mathbb R^2$, we set
	\[
	e_2(\ell) := H^\bullet(\mathrm{Conf}_\ell(\C)).
	\]
	By a theorem of Arnold~\cite{arnold2013cohomology}, $e_2(\ell)$ is the graded commutative algebra generated by degree-1 elements
	\[
	\alpha_{ij} = \alpha_{ji}, \quad 1\le i\ne j\le \ell,
	\]
	subject to the Arnold relations
	\[
	\alpha_{ij}\alpha_{jk}+\alpha_{jk}\alpha_{ki}+\alpha_{ki}\alpha_{ij}=0
	\quad\text{for all distinct }i,j,k.
	\]
	Moreover, $e_2(\ell)$ can be realized as the subalgebra of
	$\Omega^\bullet(\mathrm{Conf}_\ell(\C))$ generated by the closed 1-forms
	\begin{equation}
		\label{alpha_ij_1_form_eq}
		\tilde{\alpha}_{ij}=\frac{1}{2\pi \sqrt{-1}}\frac{dz_i - dz_j}{z_i - z_j},\qquad 1\le i\ne j\le \ell.
	\end{equation}
	Henceforth, we identify $\alpha_{ij}$ with $\tilde{\alpha}_{ij}$, and regard elements of $e_2(\ell)$ as the corresponding differential forms, rather than as cohomology classes.
	
	Consider the inclusion $\iota: \mathbb D\hookrightarrow \C$ and the induced inclusion of configuration spaces. The deformation retraction of $\C$ onto $\mathbb D$ along radial lines,
	\[
	H(z,t) = \begin{cases}
		z & |z|\le 1 \\
		(1-t)z + t\frac{z}{|z|}& |z|>1,
	\end{cases} \qquad z\in\C,\ t\in[0,1],
	\]
	induces a deformation retraction of $\mathrm{Conf}_{\ell}(\C)$ onto $\mathrm{Conf}_\ell(\mathbb D)$, since 
	\[
	H(z,t)\ne H(z',t) \quad \text{ for all $t\in[0,1]$ whenever $z\neq z'$}.
	\]
	Therefore, the inclusion $\iota: \mathrm{Conf}_\ell(\mathbb D)\hookrightarrow \mathrm{Conf}_\ell(\C)$ is a homotopy equivalence. Consequently, 
	\[
	H^\bullet(\mathrm{Conf}_\ell(\mathbb D)) \cong e_2(\ell),
	\]
	and by abuse of notation, we regard $e_2(\ell)$ as the subalgebra of $\Omega^\bullet(\mathrm{Conf}_\ell(\mathbb D))$ generated by the 1-forms
	\begin{equation}\label{equation: alpha_ij differential form}
		\alpha_{ij}=\iota^*\bigg(\frac{1}{2\pi \sqrt{-1}}\frac{dz_i-dz_j}{z_i-z_j}\bigg)=\frac{1}{2\pi \sqrt{-1}}\frac{dz_i-dz_j}{z_i-z_j}\in\Omega^1(\mathrm{Conf}_\ell(\mathbb D)).
	\end{equation}
	
	\subsection{The cochain complex $B_2(\Omega(M))$} \label{subsection: Getzler Jones E_2 complex}
	As a graded vector space,
\begin{equation}
	\label{B_2_eq}
	B_2(\Omega(M)) = \bigoplus_{\ell=0}^{\infty} e_2(\ell)\otimes_{S_\ell} (\overline{\Omega(M)[2]})^{\otimes \ell},
\end{equation}
	where $\otimes_{S_\ell}$ denotes coinvariants under the natural $S_\ell$-action on $e_2(\ell)\otimes(\overline{\Omega(M)[2]})^{\otimes l}$:
	\[
	\sigma((\alpha_{i_1j_1}\cdots\alpha_{i_kj_k})\otimes\omega_1\otimes\dots\otimes\omega_\ell)=(-1)^{\epsilon(\sigma;\omega_{[\ell]})}(\alpha_{\sigma(i_1),\sigma(j_1)}\cdots\alpha_{\sigma(i_k),\sigma(j_k)})\otimes\omega_{\sigma(1)}\otimes\dots\otimes\omega_{\sigma(\ell)}
	\]
	for any $\sigma\in S_\ell$, $k\geq 1$ and $1\leq i_n<j_n\leq \ell$ ($1\leq n\leq k$), where $\epsilon(\sigma;\omega_{[\ell]})=\sum_{1\leq i<j\leq\ell,\,\sigma(i)>\sigma(j)}|\omega_i||\omega_j|$.	
	The wedge product $\wedge_M$ on $\Omega(M)$ comes into play in defining the differential on $B_2(\Omega(M))$. For this reason, one identifies $\overline{\Omega(M)}$ with the dg subalgebra $\Omega(M,p_M)\subset\Omega(M)$ of forms that vanish at $p_M$:
	\begin{equation}\label{equation:identify overline with pointed}
	\overline{\Omega(M)} \cong \Omega(M,p_M),\quad [\omega]\mapsto \omega-\omega|_{p_M}.	    
	\end{equation}

	For simplicity, denote an element 
	$\alpha\otimes_{S_\ell}(\omega_1\otimes\dots\otimes\omega_\ell)$
	by $(\alpha | \omega_1,\dots,\omega_\ell)$. The differential on 
	\[
	B_2(\Omega(M)) = \bigoplus_{\ell=0}^{\infty} e_2(\ell)\otimes_{S_\ell} (\Omega(M,p_M)[2])^{\otimes \ell}
	\]
	decomposes as 
	\[
	d  = d_M + \delta,
	\]
	where $d_M$ is induced by the de Rham differential on $\Omega(M)$, and $\delta$ accounts for the contribution of $\wedge_M$. Specifically, 
	
	\begin{align*}
		d_M(\alpha|\omega_1,\dots,\omega_\ell) & = \sum_{i=1}^\ell (-1)^{|\alpha|+|\omega_{[i-1]}|} (\alpha | \omega,\dots,\omega_{i-1},d_M\omega_i,\omega_{i+1},\dots,\omega_\ell),\\
		\delta(\alpha|\omega_1,\dots,\omega_\ell) & = \sum_{1\leq i<j \leq \ell} (-1)^{|\alpha|+|\omega_{[i-1]}|+|\omega_{[i+1,j-1]}||\omega_j|}\big(\delta_{ij}(\alpha) | \omega_1,\dots,\omega_i\wedge_M\omega_j,\dots,\widehat{\omega}_j,\dots,\omega_\ell\big),
	\end{align*}
	where for each pair $(i,j)$, the linear map
	\[
	\delta_{ij}: e_2(\ell) \to e_2(\ell-1)
	\]
	is defined as follows.
	
	First, for each $1\leq i<j\leq \ell$, define a ring homomorphism
	\[
	\rho_{ij}: e_2(\ell) \to e_2(\ell-1)
	\]
	on the generators $\alpha_{m,n}$ ($m<n$) by
	\[
	\rho_{ij}(\alpha_{m,n}) =
	\begin{cases}
		0 & (m,n)=(i,j)\\
		\alpha_{\phi(m)\phi(n)} & (m,n) \neq (i,j)
	\end{cases}
	\]
	where $\phi=\phi_{ij}:\{1,\dots,\ell\}\to\{1,\dots,\ell-1\}$ is given by
	\begin{equation}\label{equation: rho_ij relabel}
		\phi(n)=n\ (n< j),\quad \phi(j)=i,\quad \phi(n)=n-1\ (n>j).
	\end{equation}
	Geometrically, $\rho_{ij}$ can be visualized as collapsing the $j$-th point onto the $i$-th point in an ordered tuple of distinct points in $\mathbb D$, and then renumbering the remaining points to preserve the order. It follows that $\rho_{ij}$ preserves the Arnold relations. Indeed, for any triple of distinct indices $(a,b,c)$, if $(i,j)$ appears in the corresponding Arnold relation, say $(a,b)=(i,j)$, then the terms $\alpha_{ab}\alpha_{bc}$ and $\alpha_{ca}\alpha_{ab}$ are killed by $\rho_{ij}$, while the remaining term $\alpha_{bc}\alpha_{ca}$ becomes $\alpha_{\phi(b)\phi(c)}\alpha_{\phi(c)\phi(a)} = \alpha_{i\phi(c)}\alpha_{\phi(c)i} = \alpha_{i\phi(c)}^2 = 0$ since $|\alpha_{i\phi(c)}|=1$. If $(i,j)$ does not appear, then the indices $a,b,c$ are relabeled by $\phi$, and the three-term sum vanishes by the Arnold relation for the resulting indices. Hence $\rho_{ij}$ is well-defined.\footnote{In \cite[page 67]{getzler1994operads}, $\rho_{ij}$ is defined as the endomorphism of $e_2(\ell)$ that vanishes on $\alpha_{ij}$ and preserves all other $\alpha_{i'j'}$. However, this definition appears to be a typo, since such a map does not preserve the Arnold relations.}
	
	Next, the map $\delta_{ij}: e_2(\ell)\to e_2(\ell-1)$ is defined as a degree $-1$ derivation  over $\rho_{ij}$, given on generators by
	\[
	\delta_{ij}(\alpha_{m,n}) =
	\begin{cases}
		1 & (m,n)=(i,j) \\
		0 & (m,n)\neq(i,j)
	\end{cases}
	\]
	and extended to all elements via the graded Leibniz rule over $\rho_{ij}$:
	\[
	\delta_{ij}(xy) = \delta_{ij}(x)\,\rho_{ij}(y) + (-1)^{|x|}\,\rho_{ij}(x)\,\delta_{ij}(y).
	\]
	This completes the definition of $\delta$.
	
	\subsection{The iterated integral map $\rho$}
	
	For $\ell\ge1$ and $\varepsilon>0$, set
	\[
	\mathbb D^\ell_\varepsilon:=\{(z_1,\dots,z_\ell)\in\mathbb D^\ell \mid |z_i-z_j|\geq \varepsilon\ (\forall i,j)\}.
	\]
	We may assume $\varepsilon$ is sufficiently small so that $\mathbb D^\ell_\varepsilon$ is a nonempty closed subset of $\mathbb D^\ell$.
	
	Given any $\alpha\in e_2(\ell)$ of degree $k$, we regard it as a $k$-form on $\mathbb D^\ell_\varepsilon\subset\mathrm{Conf}_\ell(\mathbb D)$ as in \eqref{equation: alpha_ij differential form}. Then we define a $(2\ell-k)$-current $\mathbb D^\ell(\alpha)$ on $\mathbb D^\ell$ by
	\[
	\int_{\mathbb D^\ell(\alpha)}\beta := \int_{\mathbb D^\ell} \alpha\wedge\beta, \qquad \forall\beta\in\Omega_c(\mathbb D^\ell)=\Omega(\mathbb D^\ell),
	\]
	where 
	\[
	\int_{\mathbb D^\ell} \alpha\wedge\beta = \lim_{\varepsilon \to 0}\int_{\mathbb D^\ell_\varepsilon} \alpha\wedge\beta
	\]
	is a convergent improper integral. After suitable substitutions, the convergence essentially follows from the convergence of $\int_{\mathbb D} \frac{dz}{z}$.
	
	For any $1\leq i<j \leq \ell$, define a smooth inclusion
	\[
	\kappa_{ij}:\mathbb D^{\ell-1}\to\mathbb D^\ell,\quad \kappa_{ij}(z_1,\dots,z_{\ell-1})=(z_1,\dots,z_{j-1},z_i,z_{j+1},\dots,z_{\ell-1}).
	\]
	Observe that $\kappa_{ij}^*(dz_k)=dz_{\phi(k)}$ where $\phi$ is the map \eqref{equation: rho_ij relabel}.
	
	\begin{lem}\label{lemma: current delta_ij alpha}
		For any $\ell\geq 1$, $\alpha\in e_2(\ell)$ and $\beta\in\Omega(\mathbb D^\ell)$, we have
		\[
		(-1)^{|\alpha|}\int_{\mathbb D^\ell(\alpha)} d\beta = \int_{\partial\mathbb D^\ell}\alpha\wedge\beta + \sum_{1\leq i<j \leq \ell}\int_{\mathbb D^{\ell-1}(\delta_{ij}\alpha)}\kappa_{ij}^*\beta,
		\]
	\end{lem}
	\begin{proof}
		We compute 
		\[
		(-1)^{|\alpha|} \int_{\mathbb D^\ell(\alpha)} d\beta = (-1)^{|\alpha|} \lim_{\varepsilon\to0} \int_{\mathbb D^\ell_\varepsilon} \alpha\wedge d\beta = \lim_{\varepsilon\to0} \int_{\mathbb D^\ell_\varepsilon} d(\alpha\wedge\beta) = \lim_{\varepsilon\to0} \int_{\partial\mathbb D^\ell_\varepsilon} \alpha\wedge\beta,
		\]
		where the last equality follows from the Stokes formula (Theorem \ref{integration_along_fiber_thm} with $\mathscr X=\mathrm{pt}$). The boundary of the manifold with corners $\mathbb D^\ell_\varepsilon$ is
		\[
		\partial \mathbb D^\ell_\varepsilon = (\mathbb D^\ell_\varepsilon \cap \partial \mathbb D^\ell) \sqcup \bigsqcup_{1\leq i<j\leq \ell}\partial_{ij}\mathbb D^\ell_\varepsilon,
		\]
		where
		\[
		\partial_{ij}\mathbb D^\ell_\varepsilon := \{(z_1,\dots,z_\ell)\in\mathbb D^\ell_\varepsilon \mid |z_i-z_j|=\varepsilon\}.
		\]
Notice that
		\[
		\lim_{\varepsilon\to0} \int_{\mathbb D^\ell_\varepsilon\cap \partial \mathbb D^\ell} \alpha\wedge\beta = \int_{\partial \mathbb D^\ell} \alpha\wedge\beta
		\]
		as a convergent improper integral. Hence, it suffices to show that for any $1\leq i<j \leq \ell$ and $\alpha=\alpha_{i_1j_1}\wedge\dots\wedge\alpha_{i_kj_k}$, where $1\leq i_n< j_n\leq \ell$ ($1\leq n\leq k$), there holds
		\[
		\lim_{\varepsilon\to0} \int_{\partial_{ij}\mathbb D^\ell_\varepsilon} \alpha\wedge\beta = \lim_{\varepsilon\to0} \int_{\mathbb D^{\ell-1}_\varepsilon} \delta_{ij}(\alpha)\wedge \kappa_{ij}^*\beta.
		\]   
		For $\alpha\neq0$, the pairs $(i_1,j_1),\dots,(i_k,j_k)$ must be mutually distinct. If $(i,j)\neq(i_n,j_n)$ for any $n$, then the equality holds as both sides are zero. If $(i,j)=(i_n,j_n)$ for some $n$, we compute
		\begin{align*}
			& \int_{\partial_{ij}\mathbb D^\ell_\varepsilon} \alpha\wedge\beta \\
			= \; & \int_{\mathbb D^\ell_\varepsilon\cap\{|z_i-z_j|=\varepsilon\}} \alpha_{i_1j_1}\wedge\dots\wedge\alpha_{i_kj_k}\wedge\beta(z_1,\dots,z_\ell) \\
			= \; & \int_{\mathbb D^\ell\cap\{|z_i-z_j|=\varepsilon\}} \alpha_{i_1j_1}\wedge\dots\wedge\alpha_{i_kj_k}\wedge\beta(z_1,\dots,z_\ell)\; + \; O(\varepsilon) \\
			= \; & \int_{\mathbb D^{j-1}\times [0,2\pi]\times\mathbb D^{\ell-j}}\alpha_{i_1j_1}\wedge\dots\wedge \frac{d\theta}{2\pi }\wedge\dots\wedge\alpha_{i_kj_k}\wedge\beta(z_1,\dots,z_{j-1},z_i+\varepsilon e^{\sqrt{-1}\theta},z_{j+1},\dots,z_\ell) \; + \; O(\varepsilon) \\
			= \; & \int_{\mathbb D^{\ell-1}_{z_1,\dots,\widehat{z_j},\dots,z_\ell}} (-1)^{n-1} \alpha_{i_1j_1}\wedge\dots\wedge\widehat{\alpha_{i_nj_n}}\wedge\dots\wedge\alpha_{i_kj_k}\wedge\beta(z_1,\dots,z_{j-1},z_i,z_{j+1},\dots,z_\ell) \; + \; O(\varepsilon) \\
			= \; & \int_{\mathbb D^{\ell-1}} \delta_{ij}(\alpha)\wedge\kappa_{ij}^*\beta  \; + \; O(\varepsilon) \quad = \quad \int_{\mathbb D^{\ell-1}_{\varepsilon}} \delta_{ij}(\alpha)\wedge\kappa_{ij}^*\beta  \; + \; O(\varepsilon).
		\end{align*}
		This proves the lemma.
	\end{proof}
	
	For a differentiable space $\mathscr Y$, define the integration of $\Omega(\mathscr Y)$-valued forms on $\mathbb D^\ell$ (Definition \ref{defn: form-valued functions}) against the current $D^\ell(\alpha)$ along the fibers
	\[
	\int_{\mathbb D^\ell(\alpha)} : \Omega^{\bullet}(\mathbb D ^\ell \times\mathscr Y) \to \Omega^{\bullet-2\ell+|\alpha|}(\mathscr Y),
	\]
	by
	\[
	\int_{\mathbb D^\ell(\alpha)}\nu := \int_{\mathbb D^\ell}\alpha\wedge\nu = \lim_{\varepsilon\to0}\int_{\mathbb D^\ell_{\varepsilon}}\alpha\wedge\nu,
	\]
	where $\alpha\wedge\nu$ is viewed as a $\Omega(\mathscr Y)$-valued form on $\mathbb D^\ell_{\varepsilon}\subset \mathbb D$, and $\int_{\mathbb D^{\ell}_{\varepsilon}}$ is the version of integration along the fibers in Section~\ref{section:integration along fibers}.
	
	For 
	\[
	\mathscr Y = \mathscr X_0 = \mathrm{Map}((\mathbb D,S^1),(M,p_M))
	\]
	the iterated integral map $\rho$ is induced by a sequence of linear maps
	\[
	\rho_\ell: e_2(\ell)\otimes\Omega(M,p_M)^{\otimes \ell} \to \Omega(\mathscr X_0)
	\]
	of internal degree $-2\ell$, defined by
	\[
	\rho_\ell(\alpha\otimes\omega_1\otimes\dots\otimes\omega_{\ell}) = 
	(-1)^{\ell+|\omega_{[\ell]}|} \int_{\mathbb D^\ell(\alpha)} \Ev_\ell^* \left(\bigtimes_{i=1}^\ell \omega_i\right)
	\]
	where $\Ev_\ell$ is the obvious evaluation map
	\[
	\Ev_\ell: \mathscr X_0 \times \mathbb D^\ell \to M^\ell.
	\]
	It follows from Lemma \ref{lemma: current delta_ij alpha} and Theorem \ref{integration_along_fiber_thm} that $\rho$ is a chain map; see also \cite[Lemma 6.12]{getzler1994operads}.
	
	\begin{thm}\label{thm: Getzler Jones double loop space}(See \cite[Theorem 6.13]{getzler1994operads}) If $M$ is a 2-connected manifold, the iterated integral map 
		\[
		\rho: B_2(\Omega(M)) \to \Omega(\mathscr X_0)
		\] is a quasi-isomorphism. (More generally, a weak equivalence of commutative co-Hopf 2-algebras.)       
	\end{thm}
	
	\begin{rem}
		In \cite{getzler1994operads}, Getzler-Jones used the dg subalgebra $\mathcal{A}(M)\subset\Omega(M)$ (see Lemma \ref{lemma: sub dga (M,N)}) in place of $\Omega(M)$ to establish Theorem \ref{thm: Getzler Jones double loop space}. By a similar argument as Lemma \ref{lemma: sub dga (M,N)}, one can show that the natural inclusion $B_2(\mathcal{A}(M))\to B_2(\Omega(M))$ is a quasi-isomorphism. 
	\end{rem}
	
	\subsection{Compatibility of $\rho$ and $J$}
	\label{s_compatible_GJ}
	For each $\ell\geq0$, the inclusion 
	\[
	\mathbb R = H^0(\mathrm{Conf}_\ell(\mathbb D)) \hookrightarrow H^\bullet(\mathrm{Conf}_\ell(\mathbb D)) = e_2(\ell)
	\]
	induces an $S_\ell$-equivariant linear map
	\[
	T_\ell: \Omega(M)^{\otimes\ell} \to e_2(\ell)\otimes\Omega(M)^{\otimes \ell}.
	\]
	Hence, $\{T_\ell\}_{\ell\geq0}$ induces a natural map
	\begin{equation}\label{equation: OCHA to Getzler-Jones}
		T : C(\Omega(M);\Omega(\pt)) = \bigoplus_{\ell\geq0} \overline{\Omega(M)}^{\wedge\ell} \to \bigoplus_{\ell\geq0} e_2(\ell)\otimes_{S_\ell}\Omega(M,p_M)^{\otimes \ell} = B_2(\Omega(M))
	\end{equation}
    where we use the identification \eqref{equation:identify overline with pointed}.
	Since $\delta$ vanishes on $\mathbb R\subset e_2(\ell)$, $T$ is a cochain map.
	By construction,
	\begin{equation}\label{equation: J compatible with rho}
		J = \rho\circ T: C(\Omega(M);\Omega(\pt)) \to \Omega(\mathscr X_0),
	\end{equation}
	where $J$ is as in \eqref{equation: J N=pt}.

	\bibliographystyle{abbrv}
	\bibliography{ref}

\end{document}